\documentclass[a4paper]{amsart}
\usepackage{amsmath}
\usepackage{amssymb}
\usepackage{amsthm}
\usepackage[all]{xy}

\begin{document}

\theoremstyle{definition}
\newtheorem{fliess}{}[section]
\newtheorem{defin}[fliess]{Definition}
\newtheorem{expl}[fliess]{Example}
\newtheorem{remark}[fliess]{Remark}
\theoremstyle{plain}
\newtheorem{thm}[fliess]{Theorem}
\newtheorem{prop}[fliess]{Proposition}
\newtheorem{lem}[fliess]{Lemma}
\newtheorem{cor}[fliess]{Corollary}
\newtheorem{claim}{Claim}

\title[Global structure of special cycles on unitary Shimura varieties]{On the global structure of special cycles on unitary Shimura varieties}
\author{Nicolas Vandenbergen}

\begin{abstract}
In this paper, we study the reduced loci of special cycles on local models of the Shimura variety for $GU(1,n-1)$. Those special cycles are defined by Kudla and Rapoport in \cite{kr}. We explicitly compute the global structure of the reduced locus of a single special cycle, as well as of an arbitrary intersection of special cycles, in terms of Bruhat-Tits theory. Furthermore, as an application of our results, we prove the connectedness of arbitrary intersections of special cycles, as conjectured by Kudla and Rapoport (\cite{kr}, Conj.\ 1.3).
\end{abstract}

\maketitle

\section{Introduction}

\noindent \bfseries Motivation. \normalfont A local analogue for the Shimura variety for $GU(1,n-1)$ has been defined in Vollaard's paper \cite{voll} as a formal moduli scheme $\mathcal{N}(1,n-1)$ of $p$-divisible groups with certain additional structures. In the subsequent paper \cite{vw}, Vollaard and Wedhorn give an explicit description of its reduced locus $\mathcal{N}_\textnormal{red}$ by stratifying it with locally closed subvarieties over $\overline{\mathbb{F}}_p$. The strata will be referred to as ``Bruhat-Tits strata'', because they are in bijection to the set of vertices of the Bruhat-Tits building of a certain special unitary group over $\mathbb{Q}_p$. In their paper \cite{kr}, Kudla and Rapoport define special homomorphisms as elements of a certain hermitian $\mathbb{Q}_{p^2}$-vector space $(\mathbb{V},h)$. Given $x\in\mathbb{V}$, they define its associated special cycle $\mathcal{Z}(x)$ as a certain formal subscheme of $\mathcal{N}(1,n-1)$. One should think of those special cycles as the local analogues of special arithmetic cycles on the Shimura variety for $GU(1,n-1)$; the latter, as well as the link between the local and the global situation, are explained in \cite{kr3}.

Kudla and Rapoport show (\cite{kr}, Thm.\ 1.1(i)) that, given $n$ special homomorphisms $x_1,\dots,x_n$, the reduced locus of the intersection $\mathcal{Z}(\underline{x}):=\mathcal{Z}(x_1),\dots,\mathcal{Z}(x_n)$ of their associated special cycles is a union of Bruhat-Tits strata in $\mathcal{N}_\textnormal{red}$ under the assumption that the fundamental matrix $T(\underline{x}):=(h(x_i,x_j))_{i,j}$ is nonsingular. They also compute the dimension of $\mathcal{Z}(\underline{x})_\textnormal{red}$ as a function of the fundamental matrix (\cite{kr}, Thm.\ 1.1(ii)) and gave a condition on the fundamental matrix which is necessary and sufficient for $\mathcal{Z}(\underline{x})_\textnormal{red}$ to be irreducible (\cite{kr}, Thm.\ 1.1(iii)). In the case of proper intersections, i.e.\ in the zero-dimensional case, they show connectedness and compute the intersection multiplicity, relating it to representation densities of hermitian forms (\cite{kr}, Thm.\ 1.1(iv)).

Kudla and Rapoport also conjectured (\cite{kr}, Conj.\ 1.3) that the reduced locus of an improper intersection $\mathcal{Z}(\underline{x})$ of $n$ special cycles is connected, and that the relation between intersection multiplicities and representation densities should also hold in this case.

The aim of this paper is to prove the first part of this conjecture and, more generally, to describe the reduced locus of intersections of special cycles. We will give a full description of the reduced locus $\mathcal{Z}(x)_\textnormal{red}$ for any $x\in\mathbb{V}$, as well as of the reduced locus of an intersection of arbitrarily many special cycles.\\
\\
\bfseries Main results. \normalfont We now recall the definitions from \cite{voll}, \cite{vw} and \cite{kr} necessary to state our results. We fix an odd prime $p$ and a positive integer $n$. Let $\mathbb{F}:=\overline{\mathbb{F}}_p$. We denote by $W:=W(\mathbb{F})$ the corresponding ring of Witt vectors and by $W_\mathbb{Q}:=W\otimes_\mathbb{Z}\mathbb{Q}$ its quotient field. The Frobenius lifts to an automorphism of $W$, resp.\ $W_\mathbb{Q}$, which we denote by $\sigma$. There are two embeddings of $\mathbb{Z}_{p^2}$, resp.\ $\mathbb{Q}_{p^2}$, into $W$, resp. $W_\mathbb{Q}$, which we denote by $\varphi_0$ and $\varphi_1=\sigma\circ\varphi_0$.

The moduli scheme $\mathcal{N}(1,n-1)$ on which we work is described as follows. We fix a triple $(\mathbb{X},\iota_\mathbb{X},\lambda_\mathbb{X})$, where $\mathbb{X}$ is a supersingular $p$-divisible group of dimension $n$ and height $2n$ over $\mathbb{F}$, equipped with an action $\iota_\mathbb{X}:\mathbb{Z}_{p^2}\rightarrow\textnormal{End}(\mathbb{X})$ satisfying the signature condition $(1,n-1)$, i.e.\
\[ \textnormal{charpol}(\iota_\mathbb{X}(\alpha),\textnormal{Lie}\ \mathbb{X})=(T-\varphi_0(\alpha))(T-\varphi_1(\alpha))^{n-1}\in \mathbb{F}\left[T\right]\ \forall\ \alpha\in\mathbb{Z}_{p^2}, \]
and with a $p$-principal polarization $\lambda_\mathbb{X}$ for which the Rosati involution $\ast$ satisfies $\iota_\mathbb{X}(\alpha)^\ast=\iota_\mathbb{X}(\alpha^\sigma)$. Note that such a triple always exists and is unique up to isogeny.

Let $\textnormal{Nilp}_W$ be the category of $W$-schemes on which $p$ is locally nilpotent. Let
\[ \mathcal{N}=\mathcal{N}(1,n-1):\textnormal{Nilp}_W\rightarrow\textnormal{Sets} \]
be the functor which associates to a scheme $S\in\textnormal{Nilp}_W$ the set of isomorphism classes of quadruples $(X,\iota_X,\lambda_X,\rho_X)$, where $X$ is a $p$-divisible group over $S$, where $\iota_X$ and $\lambda_X$ are as above, and where
\[ \rho_X:X\times_W \mathbb{F}\rightarrow\mathbb{X}\times_\mathbb{F} \overline{S} \]
is a quasi-isogeny of height 0 compatible with the additional structures imposed. Here, $\overline{S}=S\times_W\mathbb{F}$ denotes the special fibre of $S$. (See Section 2 for a precise definition of $\mathcal{N}$.)

$\mathcal{N}$ is represented by a formal scheme which we also denote by $\mathcal{N}$. This formal scheme is separated, locally formally of finite type over $W$, and formally smooth of dimension $n-1$ over $W$. The underlying reduced scheme $\mathcal{N}_\textnormal{red}$ is a singular scheme of dimension $\lfloor(n-1)/2\rfloor$ over $\mathbb{F}$.

In order to explain our results, we have to recall some of the results of Vollaard and Wedhorn (\cite{voll}, \cite{vw}) on the structure of $\mathcal{N}_\textnormal{red}$. By Dieudonn\'e theory the triple $(\mathbb{X},\iota_\mathbb{X},\lambda_\mathbb{X})$ defines a $n$-dimensional non-degenerate hermitian vector space $(C,h)$ over $\mathbb{Q}_{p^2}$ for which $\textnormal{ord}\,\textnormal{det}(C,h)$ is odd. Note that this condition determines $(C,h)$ up to isomorphism. A \itshape vertex \normalfont is then by definition a $\mathbb{Z}_{p^2}$-lattice $\Lambda$ in $C$ satisfying
\[ p\Lambda\subseteq\Lambda^\sharp{\subset}\Lambda, \]
where $\Lambda^\sharp:=\{x\in C\ |\ h(x,\Lambda)\subseteq\mathbb{Z}_{p^2}\}$ is the dual lattice of $\Lambda$. The set $\mathcal{L}$ of all vertices is made into a simplicial complex in the following way: Two vertices $\Lambda$, $\tilde\Lambda$ are by definition neighbours in the simplicial complex $\mathcal{L}$ if and only if $\Lambda\subset\tilde\Lambda$ or $\tilde\Lambda\subset\Lambda$. Then Vollaard shows that $\mathcal{L}$ is isomorphic to the simplicial complex of the Bruhat-Tits building $\mathfrak{B}(SU(C,ph),\mathbb{Q}_p)$, hence our use of the term ``vertex''. To each vertex $\Lambda\in\mathcal{L}$, Vollaard and Wedhorn associate a closed irreducible subvariety of $\mathcal{N}_\textnormal{red}$, the \itshape closed Bruhat-Tits stratum \normalfont $\mathcal{N}_\Lambda$. They show that $\mathcal{N}_\textnormal{red}$ is covered by the $\mathcal{N}_\Lambda$.

We furthermore write $t_{\max}$ for the maximal odd integer less than or equal to $n$, and $\mathcal{L}^{\max}$ for the set of vertices of maximal type in $\mathcal{L}$ ; that is, $\mathcal{L}^{\max}$ is the set of vertices $\Lambda\in\mathcal{L}$ for which the length of the quotient $\mathbb{Z}_{p^2}$-module $\Lambda/\Lambda^\sharp$ (which is always an odd integer between 1 and $n$) equals $t_{\max}$. Vollaard and Wedhorn show that the closed Bruhat-Tits strata $\mathcal{N}_\Lambda$ corresponding to these vertices are precisely the irreducible components of $\mathcal{N}_\textnormal{red}$.

We now define special cycles. Let $(\mathbb{Y},\iota_\mathbb{Y},\lambda_\mathbb{Y})$ be the basic triple over $\mathbb{F}$ used in the definition of $\mathcal{N}(1,0)$. Let $(\overline{\mathbb{Y}},\iota_{\overline{\mathbb{Y}}},\lambda_{\overline{\mathbb{Y}}})=(\mathbb{Y},\iota_\mathbb{Y}\circ\sigma,\lambda_\mathbb{Y})$. The space of \itshape special homomorphisms \normalfont is the $\mathbb{Q}_{p^2}$-vector space
\[ \mathbb{V}:=\textnormal{Hom}_{\mathbb{Z}_{p^2}}(\overline{\mathbb{Y}},\mathbb{X})\otimes_\mathbb{Z}\mathbb{Q}, \]
endowed with the non-degenerate $\mathbb{Q}_{p^2}$-valued hermitian form $h$ given by
\[ h(x,y):=\lambda_{\overline{\mathbb{Y}}}^{-1}\circ y^\vee\circ\lambda_\mathbb{X}\circ x\in\textnormal{End}_{\mathbb{Z}_{p^2}}(\overline{\mathbb{Y}})\otimes_\mathbb{Z}\mathbb{Q}\cong\mathbb{Q}_{p^2}, \] 
where the last isomorphism is induced by $\iota^{-1}$. Kudla and Rapoport show that $(\mathbb{V},h)\cong(C,h)$.

The pair $(\overline{\mathbb{Y}},\iota_{\overline{\mathbb{Y}}})$ admits a unique lift to $W$ as a formal $\mathbb{Z}_{p^2}$-module, which we denote by $(\overline{Y},\iota_{\overline{Y}})$. Now for a special homomorphism $x$, we define the \itshape special cycle \normalfont $\mathcal{Z}(\underline{x})$ to be the subfunctor of $\mathcal{N}$ such that $\mathcal{Z}(x)(S)$ consists of the tuples $(X,\iota_X,\lambda_X,\rho_X)\in\mathcal{N}(S)$ for which the quasi-homomorphism
\[ \rho_X^{-1}\circ x:(\overline{Y}\times_W S)\times_W\mathbb{F}= \overline{\mathbb{Y}}\times_\mathbb{F}\overline{S}\longrightarrow X\times_W\mathbb{F} \]
lifts to a $\mathbb{Z}_{p^2}$-linear homomorphism $(\overline{Y}\times_W S)\rightarrow X$. Finally, we associate to a tuple $(x_1,\dots,x_m)$ of special homomorphisms its \itshape fundamental matrix\normalfont,
\[ T(x_1,\dots,x_m):=(h(x_i,x_j))_{i,j}\in\textnormal{Herm}_m(\mathbb{Q}_{p^2}). \]

We now state our results. First, we generalize Thm.\ 1.1(i) and (ii) of \cite{kr} to the case of arbitrary intersections of special cycles.

\begin{prop} \label{dimprime}
The reduced locus $\mathcal{Z}(x_1,\dots,x_m)_\textnormal{red}$ of any intersection 
\[ \mathcal{Z}(x_1,\dots,x_m):=\mathcal{Z}(x_1)\cap\dots\cap\mathcal{Z}(x_m) \]
of special cycles is a union of Bruhat-Tits strata. If $T(x_1,\dots,x_m)$ is non-singular with integral entries, then $\mathcal{Z}(x_1,\dots,x_m)_\textnormal{red}$ is pure of dimension 
\[ \dim\,\mathcal{Z}(x_1,\dots,x_m)_\textnormal{red}=\lfloor (n-\textnormal{rk}(\textnormal{red}(T))-1)/2\rfloor, \]
where $\textnormal{red}(T)$ is the image of $T$ in $\textnormal{Herm}_m(\mathbb{F}_{p^2})$.
\end{prop}

Therefore, the reduced locus $\mathcal{Z}(x_1,\dots,x_m)_\textnormal{red}$ of the intersection of the special cycles associated to $x_1,\dots,x_m\in\mathbb{V}$ is determined by the following simplicial subcomplex of $\mathcal{L}$:
\[ \mathcal{S}(x_1,\dots,x_m):=\{\Lambda\in\mathcal{L}\ |\ x_i\in\Lambda^\sharp\ \forall\ 1\leq i\leq m\}. \]
Here, we view special homomorphisms as vectors in $C$ via the isomorphism $(\mathbb{V},h)\cong(C,h)$.

We then turn to computing the simplicial complex $\mathcal{S}(x)$ for a single special homomorphism $x$. We write $r(x):=v(h(x,x))$, where $v$ is the discrete valuation on $\mathbb{Q}_p$, and call this number the \itshape valuation \normalfont of $x$. We may assume that $r(x)$ is a nonnegative integer, thus the dimension computation of Prop.\ \ref{dimprime} applies.

We will give a recursive formula for $\mathcal{S}(x)$, using the cases where $r(x)=0$ and $r(x)=1$ as base cases. The recursion step will give an explicit formula for $\mathcal{S}(x)$ in terms of $\mathcal{S}(p^{-1}x)$ if $r(x)>1$ (note that $r(p^{-1}x)=r(x)-2$).

In order to state our results for the two base cases, we view the special homomorphism $x$ as an element of $C$ and write $C_\perp:=(\mathbb{Q}_{p^2}\cdot x)^\perp$, where the orthogonal complement is taken w.r.t.\ $h$.

\begin{thm} \label{prime}
Let $x\in\mathbb{V}$ be a special homomorphism of nonnegative valuation.
\begin{enumerate}
\item (\cite{kr2}, Prop.\ 5.2) If $r(x)=0$, then there is an isomorphism of simplicial complexes
\[ \mathcal{S}(x)\cong\mathfrak{B}(SU(C_\perp,ph),\mathbb{Q}_p). \]
\item If $r(x)=1$, then there is an injective morphism of simplicial complexes
\[ \Phi:\mathfrak{B}(SU(C_\perp,ph),\mathbb{Q}_p)\rightarrow \mathcal{S}(x), \]
which is surjective in maximal type, i.e.\ $\mathcal{S}(x)\cap\mathcal{L}^{\max}$ is contained in the image of $\Phi$.
\item If $r(x)\geq 2$, then 
\[ \mathcal{S}(x)=\{\Lambda\in\mathcal{L}\ |\ \exists\ \tilde\Lambda\in\mathcal{L},\hat\Lambda\in \mathcal{S}(p^{-1}x):\ \Lambda\subseteq\tilde\Lambda,\ \hat\Lambda\subseteq\tilde\Lambda\}. \]
\end{enumerate}
\end{thm}

The first part is due to Kudla and Rapoport, while the other two results are new. Kudla and Rapoport explicitly compare $\mathcal{Z}(x)(\mathbb{F})$ to $\mathcal{N}(1,n-2)(\mathbb{F})$ for $x$ of valuation 0 in their unpublished notes \cite{kr2}, introducing a certain stratification of $\mathcal{Z}(x)(\mathbb{F})$. The Bruhat-Tits building $\mathfrak{B}(SU(C_\perp,ph),\mathbb{Q}_p)$ comes into play because it encodes the incidence relation of Bruhat-Tits strata in $\mathcal{N}(1,n-2)_\textnormal{red}$. I should note that Kudla and Rapoport actually prove a stronger result, namely that $\mathcal{Z}(x)\cong\mathcal{N}(1,n-2)$ \itshape as formal schemes \normalfont (\cite{kr}, Prop.\ 9.2).

In the case $r(x)=1$, the order of determinant of $h$ on $C_\perp$ is even and thus $\mathfrak{B}(SU(C_\perp,ph),\mathbb{Q}_p)$ does not encode the incidence relation of Bruhat-Tits strata in $\mathcal{N}(1,n-2)_\textnormal{red}$. In our paper, $\mathfrak{B}(SU(C_\perp,ph),\mathbb{Q}_p)$ will not arise from a stratification of a moduli scheme. Instead, we will interpret $\mathfrak{B}(SU(C_\perp,ph),\mathbb{Q}_p)$ as the simplicial complex of lattices $\Lambda_\perp$ in $C_\perp$ satisfying a chain condition $p\Lambda_\perp\subseteq\Lambda_\perp^\sharp\subseteq\Lambda_\perp$ and use a ``stratification'' of the set $\mathcal{L}$ similar to the one on $\mathcal{Z}(x)(\mathbb{F})$ used by Kudla and Rapoport. The key point is to show that any vertex $\Lambda\in \mathcal{S}(x)$ which is of maximal type in $\mathcal{L}$ decomposes orthogonally as $\Lambda=p^{-1}\mathbb{Z}_{p^2}x\oplus(\Lambda\cap C_\perp)$. From this we then deduce Thm.\ \ref{prime}(2), which, together with pure-dimensionality of $\mathcal{Z}(x)_\textnormal{red}$, shows that $\mathcal{S}(x)$ can explicitly be computed from $\mathfrak{B}(SU(C_\perp,ph),\mathbb{Q}_p)$.

To prove Thm.\ \ref{prime}(3), we generalize the approach of Terstiege, who in his paper \cite{ter} gives an explicit formula for $\mathcal{S}(x)$ for arbitrary $x$ in the first non-trivial case $n=3$ by using ad hoc methods for $x$ of valuation 0 or 1 and showing an inductive formula of the type claimed in (3). The generalization of Terstiege's proof is quite straightforward in the case where $n$ is odd. If $n$ is even, some extra work has to be done to show that $x\in\Lambda^\sharp$ implies $p^{-1}x\in\Lambda$ for vertices $\Lambda$ of maximal type.

Finally, for any $m$ and any $x_1,\dots,x_m\in\mathbb{V}$, we explicitly compute the simplicial complex $\mathcal{S}(x_1,\dots,x_m)$ of vertices whose associated Bruhat-Tits strata are contained in $(\mathcal{Z}(x_1)\cap\dots\cap\mathcal{Z}(x_m))_\textnormal{red}$. For this, we assume (w.l.o.g., see \ref{fleshwound}) that the $x_i$ are perpendicular to each other w.r.t.\ $h$, that all valuations $r(x_i)$ are nonnegative integers and that the $x_i$ are ordered increasingly by valuation. The rough idea is to apply the formulae of Thm.\ \ref{prime} alternatingly. More precisely:

Let $r$ be any nonnegative integer. Write $m_r:=\textnormal{max}\{i\ |\ r(x_i)<r\}$. Denote by $C_r$ the orthogonal complement of the subspace of $C$ spanned by $x_1,\dots,x_{m_r}$. Set 
\[ \mathcal{L}_r:=\{ \Lambda\subset C_r\ \textnormal{a}\ \mathbb{Z}_{p^2}\textnormal{-lattice}\ |\ p\Lambda\subseteq\Lambda^\sharp\subseteq\Lambda\}. \]
By the results of Vollaard, in particular Thm.\ 3.6 of \cite{voll}, we can endow the set $\mathcal{L}_r$ with a simplicial complex structure as we did for $\mathcal{L}$, and then we have an isomorphism $\mathcal{L}_r\cong\mathfrak{B}(SU(C_r,ph),\mathbb{Q}_p)$ of simplicial complexes. Note that, if $m=n$ (i.e.\ in the case considered in \cite{kr}) and $r>r(x_m)$, then $C_r$ is the zero space and thus the simplicial complex $\mathcal{L}_r$ consists of a single point.

We then show that, for any $s\geq 0$, there are injections of simplicial complexes 
\[ \Phi_s:\mathcal{L}_{2s+2}\rightarrow\mathcal{L}_{2s}. \]
To do this, we basically iterate the maps $\Phi:\mathfrak{B}(SU(C_\perp,ph),\mathbb{Q}_p)\rightarrow\mathcal{L}$ constructed in the proofs of Thm.\ \ref{prime}(1) and (2). We also show that the $\Phi_s$ have properties ``similar'' to those maps and furthermore that the procedure of the proof of Thm.\ \ref{prime}(3) generalizes to intersections of special cycles and to arbitrary order of determinant of $h$. Putting together, we obtain:

\begin{thm}\label{vielexprime}
Let $\underline{x}=(x_1,\dots,x_m)\in\mathbb{V}^m$ satisfy our assumptions above. Then $\mathcal{S}(\underline{x})$ is equal to the simplicial complex $S_0$ computed by the following algorithm:
\begin{enumerate}
\addtocounter{enumi}{-1}
\item Set $s:=\lfloor r(x_m)/2\rfloor$. Set $S_{s+1}':=\mathcal{L}_{2s+2}$.\smallskip
\item Set $S_s:=\{ \Lambda\in\mathcal{L}_{2s}\ |\ \exists\ \tilde\Lambda\in\Phi_s(S_{s+1}'):\ \Lambda\subseteq\tilde\Lambda \}$.\smallskip
\item Set $S_s':=\{ \Lambda\in\mathcal{L}_{2s}\ |\ \exists\ \tilde\Lambda\in\mathcal{L}_{2s},\hat\Lambda\in S_s:\ \Lambda\subseteq\tilde\Lambda,\ \hat\Lambda\subseteq\tilde\Lambda \}$.\smallskip
\item If $s=0$, stop, else replace $s$ by $s-1$ and go to Step 1.
\end{enumerate}
\end{thm}

Using the connectedness of the Bruhat-Tits buildings $\mathfrak{B}(SU(C_\perp,ph),\mathbb{Q}_p)$ in the case $m=1$, resp.\ $\mathfrak{B}(SU(C_{2s+2},ph),\mathbb{Q}_p)$ in the case $m>1$, a closer look at the explicit descriptions of $\mathcal{S}(x)$, resp.\ $\mathcal{S}(x_1,\dots,x_m)$, then proves our main theorem:

\begin{thm}\label{holygrail}
Let $m$ be any positive integer. Let $x_1,\dots,x_m\in\mathbb{V}$ be any set of special homomorphisms. Then the intersection $\mathcal{Z}(x_1,\dots,x_m)_\textnormal{red}$ of the reduced loci of the associated special cycles is connected.
\end{thm} 

\noindent \bfseries Structure. \normalfont The layout of this paper is as follows. We recall in Section 2 the definition of the moduli space $\mathcal{N}=\mathcal{N}(1,n-1)$ and the description of its $\mathbb{F}$-valued points as lattices in an isocrystal with certain additional structures. We also define the hermitian $\mathbb{Q}_{p^2}$-vector space $C$. In Section 3, we review the construction of the Bruhat-Tits stratification. The results of those sections are cited from Vollaard's paper \cite{voll} and her paper with Wedhorn \cite{vw}.

In Section 4, we define special homomorphisms and special cycles, review the results of Kudla and Rapoport, prove Prop.\ \ref{dimprime} (i.e.\ generalize Thm.\ 1.1(i) and (ii) of \cite{kr} to arbitrary intersections) and show that the connectedness conjecture of Kudla and Rapoport (i.e.\ Thm. \ref{holygrail}) can be reformulated in terms of Bruhat-Tits theory.

In Sections 5-8, we deal with the case of a single special homomorphism, i.e.\ of $m=1$. Section 5 is preparatory to the proof of our main results and basically an elaboration of Sections 6 and 7 of the notes \cite{kr2}. We give a complete proof of Lemma 5.1 of \cite{kr2}, define the Kudla-Rapoport stratification on $\mathcal{Z}(x)(\mathbb{F})$ (allowing us to treat $r=0$), and generalize it to the level of vertices (a key idea in our treatment of $r\geq 1$).

In Sections 6-8, the recursive formula for computing $\mathcal{S}(x)$ is set up, with Sections 6 and 7 devoted to the two base cases. In Section 6, we review Kudla and Rapoport's proof of Thm.\ \ref{prime}(1) and of connectedness in case of a single special homomorphism of valuation 0. Section 7 is devoted to the second base case, that is, to showing Thm.\ \ref{prime}(2) and connectedness in case of a single special homomorphism of valuation 1.

In Section 8, we deal with the recursion step, i.e.\ prove Thm.\ \ref{prime}(3) and connectedness of $\mathcal{Z}(x)_\textnormal{red}$ under the assumption that $\mathcal{Z}(p^{-1}x)_\textnormal{red}$ is connected.

Finally, Section 9 is devoted to showing our claims on intersections of special cycles, i.e.\ correctness of the algorithm of Thm.\ \ref{vielexprime} and connectedness for intersections of special cycles. We will also give a new proof of Thm.\ 1.1(iii) of \cite{kr} as an application of Thm.\ \ref{vielexprime}.\\
\\
\bfseries Acknowledgements. \normalfont I would like to thank everyone who helped me write this work. Special thanks go to M.\ Rapoport for introducing me into the subject and for lots of helpful comments and discussions. I would also like to express my gratitude to E.\ Viehmann for her interest in my work and for many helpful comments. I am also indebted to S.\ Kudla for the opportunity to use his unpublished notes \cite{kr2}, which provided essential tools used in this paper.

\section{Basic definitions}

In this section, we recall the definition of the moduli space $\mathcal{N}$ with which we will work. The exposition follows \cite{kr}. All proofs are to be found in \cite{voll}.

\begin{fliess}
We fix an odd prime $p$. Let $\mathbb{F}:=\overline{\mathbb{F}}_p$. We denote by $W:=W(\mathbb{F})$ the corresponding ring of Witt vectors and by $W_\mathbb{Q}:=W\otimes_\mathbb{Z}\mathbb{Q}$ its quotient field. The Frobenius lifts to an automorphism of $W$ resp. $W_\mathbb{Q}$, which we denote by $\sigma$. There are two embeddings of $\mathbb{Z}_{p^2}$, resp.\ $\mathbb{Q}_{p^2}$, into $W$, resp. $W_\mathbb{Q}$, which we denote by $\varphi_0$ and $\varphi_1=\sigma\circ\varphi_0$. The discrete valuation on $W_\mathbb{Q}$ will be denoted by $v$.
\end{fliess}

\begin{fliess}
We also fix a positive integer $n$ and a triple
\[ (\mathbb{X},\iota_\mathbb{X},\lambda_\mathbb{X}), \]
where $\mathbb{X}$ is a supersingular $p$-divisible group of dimension $n$ and height $2n$ over $\mathbb{F}$, on which we have an action $\iota_\mathbb{X}:\mathbb{Z}_{p^2}\rightarrow\textnormal{End}(\mathbb{X})$ satisfying the signature condition $(1,n-1)$, i.e.
\[ \textnormal{charpol}(\iota_\mathbb{X}(\alpha),\textnormal{Lie}\ \mathbb{X})=(T-\varphi_0(\alpha))(T-\varphi_1(\alpha))^{n-1}\in \mathbb{F}\left[T\right]\ \forall\ \alpha\in\mathbb{Z}_{p^2}, \]
and a $p$-principal polarization $\lambda_\mathbb{X}$ for which the Rosati involution $\ast$ satisfies 
\[ \iota_\mathbb{X}(\alpha)^\ast=\iota_\mathbb{X}(\alpha^\sigma). \]
Note that such a triple always exists and is unique up to isogeny.
\end{fliess}

\begin{fliess}
Let $\textnormal{Nilp}_W$ be the category of $W$-schemes on which $p$ is locally nilpotent. Let
\[ \mathcal{N}=\mathcal{N}(1,n-1):\textnormal{Nilp}_W\rightarrow\textnormal{Sets} \]
be the functor which associates to a scheme $S\in\textnormal{Nilp}_W$ with special fibre $\overline{S}=S\times_W\mathbb{F}$ the set of isomorphism classes of quadruples
\[ (X,\iota_X,\lambda_X,\rho_X) ,\]
where $X$ is a $p$-divisible group over $S$, endowed with an $\mathbb{Z}_{p^2}$-action $\iota_X$ satisfying the signature condition $(1,n-1)$ and a $p$-principal polarization $\lambda_X$ defined over $S$ for which the Rosati involution $\ast$ satisfies $\iota_X(\alpha)^\ast=\iota_X(\alpha^\sigma)$. Finally, 
\[ \rho_X:X\times_W \mathbb{F}\rightarrow\mathbb{X}\times_\mathbb{F} \overline{S} \]
is a quasi-isogeny of height\footnote{I should point out that the definition we are using slightly differs from the definition of $\mathcal{N}$ in \cite{voll} and \cite{kr}, as they do not impose the ``height 0'' condition. The moduli functor $\mathcal{N}$ of \cite{voll} and \cite{kr} admits a direct sum decomposition $\mathcal{N}=\coprod_i\mathcal{N}_i$, where $\mathcal{N}_i$ denotes the subfunctor of $\mathcal{N}$ of isogenies of height $i$. However, the $\mathcal{N}_i$ are either empty or isomorphic to $\mathcal{N}_0$ (\cite{voll}, Lemma 1.9 and Prop.\ 1.22).} 0 compatible with the polarizations, i.e.\ locally on $\overline{S}$ one has
\[ \rho_X^\vee\circ\lambda_\mathbb{X}\circ\rho_X=\xi\lambda_X\in\textnormal{Hom}_{\mathbb{Z}_{p^2}}(X,X^\vee)\otimes\mathbb{Q} \]
for some scalar $\xi\in\mathbb{Z}_p^\times$.

Here, two such tuples $(X,\iota_X,\lambda_X,\rho_X)$ and $(Y,\iota_Y,\lambda_Y,\rho_Y)$ are said to be isomorphic if there is an isomorphism $\alpha:X\rightarrow Y$ compatible with the $\mathbb{Z}_{p^2}$-actions and the quasi-isogenies $\rho$, such that, locally on $S$, the polarization $\alpha^\vee\circ\lambda_Y\circ\alpha$ equals $\lambda_X$ up to a scalar in $\mathbb{Z}_p^\times$.
\end{fliess}

\begin{fliess}
As explained in \cite{kr}, $\mathcal{N}$ is represented by a formal scheme which we also denote by $\mathcal{N}$. This formal scheme is separated, locally formally of finite type over $W$, and formally smooth of dimension $n-1$ over $W$.

We are interested in the geometry of the reduced locus $\mathcal{N}_\textnormal{red}$ of $\mathcal{N}$. We have $\mathcal{N}(\mathbb{F})=\mathcal{N}_\textnormal{red}(\mathbb{F})$. Theorem 4.2 of \cite{vw} states that $\mathcal{N}_\textnormal{red}$ is geometrically connected of pure dimension $\lfloor(n-1)/2\rfloor$.
\end{fliess}

\begin{fliess}
Let $\mathbb{M}$ be the covariant Dieudonn\'e module of $\mathbb{X}$. Thus $\mathbb{M}$ is a free $W$-module of rank $\textnormal{rk}(\mathbb{M})=\textnormal{ht}(\mathbb{X})=2n$ endowed with a $\sigma$-linear endomorphism $F$ (the Frobenius) and a $\sigma^{-1}$-linear endomorphism $V$ (the Verschiebung), satisfying
\[ FV=VF=p\textnormal{id}_\mathbb{M}. \]
The polarization $\lambda_\mathbb{X}$ induces a perfect skew-symmetric $W$-bilinear pairing $\langle .,.\rangle$ on $\mathbb{M}$ satisfying
\[ \langle Fx,y\rangle=\langle x,Vy\rangle^\sigma\ \ \forall\ x,y\in\mathbb{M}. \]
Furthermore, the $\mathbb{Z}_{p^2}$-action $\iota_\mathbb{X}$ induces an action of $\mathbb{Z}_{p^2}$ on $\mathbb{M}$, which we denote by $\iota$. The condition on the Rosati involution translates as
\[ \langle \iota(\alpha)x,y\rangle=\langle x,\iota(\alpha^\sigma)y\rangle. \]
The decomposition $\mathbb{Z}_{p^2}\otimes_{\mathbb{Z}_p}W\cong W\oplus W$ yields a decomposition 
\[ \mathbb{M}=\mathbb{M}_0\oplus\mathbb{M}_1 \]
into eigenspaces for the $\mathbb{Z}_{p^2}$-action. The submodules $\mathbb{M}_i$ are free of rank $n$, with $\mathbb{Z}_{p^2}$ acting by scalar multiplication via $\varphi_i$. It follows that the $\mathbb{M}_i$ are isotropic with respect to $\langle .,.\rangle$ and that both $F$ and $V$ have degree 1 w.r.t.\ this decomposition. The signature condition on the $\mathbb{Z}_{p^2}$-action translates as the chain condition
\[ p\mathbb{M}_0\overset{\text{\tiny{\itshape n-\normalfont 1}}}{\subset} F\mathbb{M}_1\overset{\text{\tiny 1}}{\subset}\mathbb{M}_0,\quad p\mathbb{M}_1\overset{\text{\tiny 1}}{\subset} F\mathbb{M}_0\overset{\text{\tiny{\itshape n-\normalfont 1}}}{\subset}\mathbb{M}_1. \]
Finally, we denote by 
\[ N:=\mathbb{M}\otimes_\mathbb{Z}\mathbb{Q} \]
the isocrystal associated to $\mathbb{X}$ and remark that, by scalar extension, $N$ is endowed with a $\sigma$-linear isomorphism $F$, a $\sigma^{-1}$-linear isomorphism $V$, a non-degenerate skew-symmetric $W_\mathbb{Q}$-bilinear form $\langle .,.\rangle$ and a $\mathbb{Q}_{p^2}$-action $\iota$, which satisfy all of the properties above. The $\mathbb{Q}_{p^2}$-action induces a decomposition
\[ N=N_0\oplus N_1 \]
into $n$-dimensional $W_\mathbb{Q}$-subspaces, with $\mathbb{Q}_{p^2}$ acting on $N_i$ by scalar multiplication via $\varphi_i$.
\end{fliess}

\begin{prop}[Vollaard, \cite{voll}, Prop.\ 1.3 and Lemma 1.5]
As a point set, $\mathcal{N}(\mathbb{F})$ may be identified with the set of $W$-lattices $M\subset N$ which are stable under $F$, $V$ and $\iota$ and which satisfy both the chain condition
\[ pM_0\overset{\text{\tiny{\itshape n-\normalfont 1}}}{\subset} FM_1\overset{\text{\tiny\normalfont 1}}{\subset}M_0,\quad pM_1\overset{\text{\tiny\normalfont 1}}{\subset} FM_0\overset{\text{\tiny{\itshape n-\normalfont 1}}}{\subset}M_1, \]
with $M_i:=M\cap N_i$, and the self-duality condition $M=M^\vee$, where
\[ M^\vee:=\{x\in N\ |\ \langle x,M\rangle\subseteq W\} \]
denotes the dual lattice of $M$ w.r.t.\ $\langle.,.\rangle$.
\end{prop}

\begin{fliess}
Define
\[ \tau:=V^{-1}F=pV^{-2}=p^{-1}F^2. \]
This defines a $\sigma^2$-linear automorphism of $N$ of degree 0, having all Newton slopes 0. Let $C=N_0^{\tau=1}$ be the space of $\tau$-invariants. This is a $\mathbb{Q}_{p^2}$-vector space of dimension $n$, and we regain $N_0$ from $C$ by base change to $W_\mathbb{Q}$.

Fix a trace zero element $\delta\in\mathbb{Z}_{p^2}^\times$, i.e.\ $\delta^\sigma=-\delta$. We define a $W_\mathbb{Q}$-sesquilinear (w.r.t.\ $\sigma$) form $h$ on $N$ by
\[ h(x,y):=p^{-1}\delta^{-1}\langle x,Fy\rangle. \]
This satisfies $h(\tau x,\tau y)=h(x,y)^{\sigma^2}$ and thus defines a $\mathbb{Q}_{p^2}$-valued hermitian form on $C$, which one checks to be non-degenerate. Note that, since $p$ is odd and $\textnormal{Nm}(\mathbb{Q}_{p^2})=\{\alpha\in\mathbb{Q}_p\ |\ v(\alpha)\equiv 0\ (2)\}$ (where $\textnormal{Nm}$ denotes the norm of $\mathbb{Q}_{p^2}/\mathbb{Q}_p$), non-degenerate hermitian forms on finite-dimensional $\mathbb{Q}_{p^2}$-vector spaces are classified by the parity of the valuation of their determinant (see \cite{jack}, Thm.\ 3.1). In our situation, the signature condition for $\mathbb{X}$ implies that the valuation of the determinant of $h$ is odd. For a $W$-lattice $L$ in $N_0$, denote by $L^\sharp$ the dual of $L$ w.r.t.\ $h$, i.e.\
\[ L^\sharp:=\{ x\in N_0\ |\ h(x,L)\subseteq W\}. \]
One checks $L^\sharp=(V^{-1}L)^\vee=FL^\vee$. Note that taking the dual w.r.t.\ $h$ is not an involution on the set of $W$-lattices in $N_0$. Indeed, one has $L^{\sharp\sharp}=\tau L$, thus the self-dual lattices in $N_0$ are the $\tau$-invariant ones, i.e.\ those which arise from $\mathbb{Z}_{p^2}$-lattices in $C$ via scalar extension.\\
\end{fliess}

\begin{prop}[Vollaard, \cite{voll}, Prop.\ 1.12] \label{nnull}
$\mathcal{N}(\mathbb{F})$ may be identified with the set of $W$-lattices $A\subset N_0$ satisfying the chain condition
\[ pA\overset{\text{\tiny\normalfont 1}}{\subset}A^\sharp\overset{\text{\tiny{\itshape n-\normalfont 1}}}{\subset}A. \]
\end{prop}

\begin{proof}
Associate to a lattice $M=M_0\oplus M_1$ in $\mathcal{N}(\mathbb{F})$ the lattice 
\[ A:=M_0 \]
in $N_0$. As $M$ is self-dual, we have $M_1=M_0^\vee=F^{-1}M_0^\sharp$. Thus, the chain condition for $M$ translates into the chain condition claimed in the proposition.\\
On the other hand, associate to a given $A\subset N_0$ with the imposed properties a lattice $M\in\mathcal{N}(\mathbb{F})$ by setting
\[ M_0:=A,\quad M_1:=F^{-1}A^\sharp. \]
These constructions are clearly inverse to each other.
\end{proof}

\section{The Bruhat-Tits stratification}

In this section, we recall Vollaard's and Wedhorn's construction of the Bruhat-Tits stratification of $\mathcal{N}_\textnormal{red}$. For proofs, see \cite{voll} and \cite{vw}.

\begin{defin}
For any odd integer $1\leq t\leq n$, set
\[ \mathcal{L}^{(t)}:=\{\Lambda\subset C\ \textnormal{a}\ \mathbb{Z}_{p^2}\textnormal{-lattice}\ |\ p\Lambda\subseteq\Lambda^\sharp\overset{\text{\tiny{\itshape t}}}{\subset}\Lambda \}. \]
$\mathcal{L}^{(t)}$ will be called the set of \itshape vertices of type $t$\normalfont.

For notational purposes, set
\[ \mathcal{L}:=\bigcup\nolimits_t \mathcal{L}^{(t)}. \]
For any vertex $\Lambda$, the integer $t(\Lambda)$ will always denote the type of $\Lambda$. Furthermore $t_{\max}$ will denote the maximal type occurring, i.e.\ the maximal odd integer less than or equal to $n$, and we let $\mathcal{L}^{\max}:=\mathcal{L}^{\left(t_{\max}\right)}$.
\end{defin}

The following proposition justifies our use of the word ``vertex''.

\begin{prop}[Vollaard, \cite{voll}, Thm.\ 3.6]\label{build}
The set $\mathcal{L}$ is canonically in bijection with the set of vertices of the Bruhat-Tits building $\mathfrak{B}(SU(C,ph),\mathbb{Q}_p)$. Given two lattices $\Lambda\neq\tilde\Lambda\in\mathcal{L}$, one of them contains the other if and only if the corresponding vertices neighbour each other in the simplicial complex of $\mathfrak{B}(SU(C,ph),\mathbb{Q}_p)$.
\end{prop}

Note that we will frequently consider $\mathcal{L}$ as a simplicial complex with the simplicial complex structure induced by this bijection.

\begin{defin}
Let $\Lambda\in\mathcal{L}$. Set
\[ \mathcal{V}(\Lambda):=\{A\in\mathcal{N}(\mathbb{F})\ |\ A\subseteq\Lambda_W\}. \]
Here, points of $\mathcal{N}(\mathbb{F})$ are viewed as $W$-lattices in $N_0$ via Prop.\ \ref{nnull}, and $\Lambda_W$ denotes the scalar extension $\Lambda\otimes_{\mathbb{Z}_{p^2}}W$, which is a $W$-lattice in $N_0$.
\end{defin}

\begin{fliess}
Let $\Lambda\in\mathcal{L}$ be of type $t$. Vollaard and Wedhorn show that $\mathcal{V}(\Lambda)$ is the set of $\mathbb{F}$-valued points of a $(t-1)/2$-dimensional irreducible smooth projective subvariety of $\mathcal{N}_\textnormal{red}$ over $\mathbb{F}$, which we denote by $\mathcal{N}_\Lambda$ (\cite{vw}, Eqn.\ 3.3.2 and Cor.\ 3.11). They also show the following facts about the $\mathcal{V}(\Lambda)$ in \cite{vw}:\smallskip\\
(1) Thm.\ 4.1(4): $\mathcal{N}(\mathbb{F})=\bigcup_\mathcal{L}\mathcal{V}(\Lambda)$.\smallskip\\
(2) Thm.\ 4.1(1): $\mathcal{V}(\Lambda)\subset\mathcal{V}(\tilde\Lambda)\Leftrightarrow\Lambda\subset\tilde\Lambda$.\smallskip\\
(3) Thm.\ 4.1(2): $\mathcal{V}(\Lambda)\cap\mathcal{V}(\tilde\Lambda)$ is nonempty if and only if $\Lambda\cap\tilde\Lambda$ is a vertex, in which case it equals $\mathcal{V}(\Lambda\cap\tilde\Lambda)$.\smallskip\\
(4) Thm.\ 4.2(2): The irreducible components of $\mathcal{N}_\textnormal{red}$ are precisely the $\mathcal{N}_\Lambda$ corresponding to the $\Lambda$ of maximal type $t_{\max}$.\smallskip

Furthermore they show (\cite{vw}, Prop.\ 4.3) that we have a stratification of $\mathcal{N}_\textnormal{red}$ by the locally-closed subvarieties
\[ \mathcal{N}_\Lambda^\circ:=\mathcal{N}_\Lambda-\bigcup_{\tilde\Lambda\subsetneq\Lambda}\mathcal{N}_{\tilde\Lambda}, \]
the so-called \itshape Bruhat-Tits stratification\normalfont.
\end{fliess}

\begin{fliess}[Vollaard, \cite{voll}, Cor.\ 2.10] \label{muff}
Let $\Lambda$ be a vertex of arbitrary type $t$. We will need a convenient description of the set of Bruhat-Tits strata contained in $\mathcal{V}(\Lambda)$ in terms of linear algebra over finite fields. To achieve this, set $V:=\Lambda/\Lambda^\sharp$. This is a $t$-dimensional $\mathbb{F}_{p^2}$-vector space endowed with a non-degenerate hermitian form $\overline{ph}$ induced by $ph$. For any odd $\tilde{t}\leq t$, the map
\begin{align*}
\{\tilde\Lambda\in\mathcal{L}^{(\tilde{t})}\ |\ \tilde\Lambda\subseteq\Lambda\} &\rightarrow \{U\subset V\ |\ \dim U=(t-\tilde{t})/2, U\textnormal{ isotropic for}\ \overline{ph}\}\\
\tilde\Lambda &\mapsto \tilde\Lambda^\sharp/\Lambda^\sharp=(\tilde\Lambda/\Lambda^\sharp)^\perp
\end{align*}
is a bijection.

In the same fashion, we get a description of the set of Bruhat-Tits strata containing $\mathcal{V}(\Lambda)$. We just consider the $(n-t)$-dimensional non-degenerate hermitian $\mathbb{F}_{p^2}$-vector space $(V',\overline{h})$, where $V':=\Lambda^\sharp/p\Lambda$ and where $\overline{h}$ is induced by $h$. We then have a bijection
\begin{align*}
\{\tilde\Lambda\in\mathcal{L}^{(\tilde{t})}\ |\ \Lambda\subseteq\tilde\Lambda\} &\leftrightarrow \{U\subset V'\ |\ \dim U=(\tilde{t}-t)/2, U\textnormal{ isotropic for}\ \overline{h}\}\\
\tilde\Lambda &\mapsto p\tilde\Lambda/p\Lambda=(\tilde\Lambda^\sharp/p\Lambda)^\perp.
\end{align*}
\end{fliess}

\begin{fliess}\label{keepyourdistance}
We will use the following generalization of the distance function on the simplicial complex of $\mathfrak{B}(SU(C),\mathbb{Q}_p)$ introduced in Terstiege's paper \cite{ter}. Let $\mathcal{S}$ be any subset of $\mathcal{L}$. Set $\mathcal{N}(\mathcal{S}):=\bigcup_{\tilde\Lambda\in\mathcal{S}} \mathcal{N}_{\tilde\Lambda}$. Let $\Lambda\in\mathcal{L}^{\max}$. We define the distance $d(\Lambda,\mathcal{S})$ to be 0 if $\Lambda\in\mathcal{S}$. Otherwise $d(\Lambda,\mathcal{S})$ is defined to be the minimal positive integer $d$ for which there exists a sequence $\Lambda_1,\dots,\Lambda_d=\Lambda$ of vertices of maximal type such that $\mathcal{N}_{\Lambda_1}$ intersects $\mathcal{N}(\mathcal{S})$ non-trivially (i.e.\ there is a vertex in $\mathcal{S}$ whose intersection with $\Lambda_1$ is again a vertex) and such that the irreducible components $\mathcal{N}_{\Lambda_i}$, $\mathcal{N}_{\Lambda_{i+1}}$ of $\mathcal{N}_\textnormal{red}$ intersect non-trivially (i.e.\ the intersections $\Lambda_i\cap\Lambda_{i+1}$ are vertices) for all $1\leq i\leq d-1$.
\end{fliess}

\section{Special cycles}

In this section, we recall the notions of special homomorphisms and special cycles, as defined by Kudla and Rapoport in \cite{kr}. We will then show Prop.\ \ref{dimprime} and reduce our main theorem \ref{holygrail} to a question about connectedness of a certain simplicial subcomplex of $\mathcal{L}$ and thus to Bruhat-Tits theory.

\begin{fliess}
Let $(\overline{\mathbb{Y}},\iota_{\overline{\mathbb{Y}}},\lambda_{\overline{\mathbb{Y}}})$ be the triple (unique up to isogeny) consisting of a supersingular $p$-divisible group $\overline{\mathbb{Y}}$ of dimension 1 and height 2 over $\mathbb{F}$, a $\mathbb{Z}_{p^2}$-action $\iota_{\overline{\mathbb{Y}}}$ on $\overline{\mathbb{Y}}$ satisfying the signature condition (0,1) (i.e.\ inducing the action by scalar multiplication via $\varphi_1$ on $\textnormal{Lie}\ \overline{\mathbb{Y}}$) and a $p$-principal polarization $\lambda_{\overline{\mathbb{Y}}}$ for which the Rosati involution satisfies $\iota_{\overline{\mathbb{Y}}}(\alpha)^\ast=\iota_{\overline{\mathbb{Y}}}(\alpha^\sigma)$. Note that $(\overline{\mathbb{Y}},\iota_{\overline{\mathbb{Y}}})$ admits a unique lift to $W$ as a formal $\mathbb{Z}_{p^2}$-module, which we denote by $(\overline{Y},\iota_{\overline{Y}})$ (\cite{kr}, Remark 2.5). 
\end{fliess}

\begin{defin}[Kudla-Rapoport, \cite{kr}, Def.\ 3.1]
The space of special homomorphisms is defined to be the $\mathbb{Q}_{p^2}$-vector space
\[ \mathbb{V}:=\textnormal{Hom}_{\mathbb{Z}_{p^2}}(\overline{\mathbb{Y}},\mathbb{X})\otimes_\mathbb{Z}\mathbb{Q}, \]
endowed with the non-degenerate $\mathbb{Q}_{p^2}$-valued hermitian form $h$ given by
\[ h(x,y):=\lambda_{\overline{\mathbb{Y}}}^{-1}\circ y^\vee\circ\lambda_\mathbb{X}\circ x\in\textnormal{End}_{\mathbb{Z}_{p^2}}(\overline{\mathbb{Y}})\otimes_\mathbb{Z}\mathbb{Q}, \]
identifying $\textnormal{End}_{\mathbb{Z}_{p^2}}(\overline{\mathbb{Y}})\otimes\mathbb{Q}$ with $\mathbb{Q}_{p^2}$ via the isomorphism $\iota_{\overline{\mathbb{Y}}}$.
\end{defin}

\begin{defin}[Kudla-Rapoport, \cite{kr}, Def.\ 3.2]
Let $1\leq m\leq n$ be an integer. Let $\underline{x}=(x_1,\dots,x_m)\in\mathbb{V}^m$ be a tuple of special homomorphisms. The special cycle associated to $\underline{x}$, denoted $\mathcal{Z}(\underline{x})$, is the formal subscheme of $\mathcal{N}$ associating to a scheme $S\in\textnormal{Nilp}_W$ the set of tuples $(X,\iota_X,\lambda_X,\rho_X)\in\mathcal{N}(S)$ for which all the quasi-homomorphisms
\[ (\overline{Y}\times_W S)\times_W\mathbb{F}= \overline{\mathbb{Y}}\times_\mathbb{F}\overline{S}\overset{x_i}{\longrightarrow}\mathbb{X}\times_\mathbb{F}\overline{S}\overset{\rho_X^{-1}}{\longrightarrow}X\times_W\mathbb{F} \]
lift to $\mathbb{Z}_{p^2}$-linear homomorphisms $\overline{Y}\times_W S\rightarrow X$.

In the case of a single special homomorphism $x\in\mathbb{V}$, we write
\[ \mathcal{Z}(x):=\mathcal{Z}((x)). \]
Given two tuples $\underline{x}=(x_1,\dots,x_m)$ and $\underline{z}=(z_1,\dots,z_l)$, we will occasionally write 
\[ \mathcal{Z}(\underline{x},\underline{z}):=\mathcal{Z}((x_1,\dots,x_m,z_1,\dots,z_l)). \]
\end{defin}

\begin{defin}
For $\underline{x}\in\mathbb{V}^m$, the fundamental matrix $T(\underline{x})$ is the hermitian matrix $(h(x_i,x_j))_{i,j}\in\textnormal{Herm}_m(\mathbb{Q}_{p^2})$.
\end{defin}

\begin{fliess}
It is obvious from the definition that $\mathcal{Z}(\underline{x})\cap\mathcal{Z}(\underline{z})=\mathcal{Z}(\underline{x},\underline{z})$ for any tuples $\underline{x},\underline{z}$ of special homomorphisms.

Kudla and Rapoport prove in \cite{kr}, Prop.\ 3.5, that for any special homomorphism $x\neq 0$, the special cycle $\mathcal{Z}(x)$ is either empty or a relative effective divisor in $\mathcal{N}$, in particular a closed formal subscheme of $\mathcal{N}$.

By Lemma 3.9 of \cite{kr}, we may identify $(\mathbb{V},h)$ with $(C,h)$, and this induces an identification
\[ \mathcal{Z}(\underline{x})(\mathbb{F})=\{A\in\mathcal{N}(\mathbb{F})\ |\ x_i\in A^\sharp\ \forall i\}, \]
where $\underline{x}=(x_1,\dots,x_m)$ is considered as a tuple of special homomorphisms on the left hand side and as a tuple of elements of $C$ on the right hand side. In particular, $\mathcal{Z}(x)_\textnormal{red}$ is nonempty if and only if all entries of the fundamental matrix are integral, i.e.\ $T(\underline{x})\in\textnormal{Herm}_m(\mathbb{Z}_{p^2})$.
\end{fliess}

\begin{fliess} \label{fleshwound}
Whenever we talk about ``a tuple $\underline{x}\in\mathbb{V}^m$ of special homomorphisms'', we will simplify notation and exclude pathological cases by making the following assumptions on $\underline{x}$ unless specified otherwise.

First, we assume that $T(\underline{x})\in\textnormal{Herm}_m(\mathbb{Z}_{p^2})$. Furthermore, the action $\iota_{\overline{\mathbb{Y}}}$ defines an action of $GL_m(\mathbb{Z}_{p^2})$ on $\overline{\mathbb{Y}}^m$, which induces a right action of $GL_m(\mathbb{Z}_{p^2})$ on $\mathbb{V}^m$ by precomposition. One has $\mathcal{Z}(\underline{x})=\mathcal{Z}(\underline{x}.g)$ for any $g\in GL_m(\mathbb{Z}_{p^2})$, i.e.\ the special cycle $\mathcal{Z}(\underline{x})$ does depend on the orbit of $\underline{x}$ under the $GL_m(\mathbb{Z}_{p^2})$-action only (\cite{kr}, Remark 3.3(i)). One then checks
\[ T(\underline{x}.g)=g^t\cdot T(\underline{x})\cdot g^{(\sigma)}. \]
As $p$ is odd and $\mathbb{Q}_{p^2}$ is unramified over $\mathbb{Q}_p$, each orbit of the right $GL_m(\mathbb{Z}_{p^2})$-action on $\textnormal{Herm}_m(\mathbb{Z}_{p^2})$ given by this formula has a representative of diagonal form (\cite{jack}, Sect.\ 7). Thus we may and will assume that $T(\underline{x})$ has diagonal form, i.e.\ that the $x_i$ are $h$-perpendicular to each other.

We may also assume that $T(\underline{x})$ is nonsingular\footnote{If $m=n$, then this is the case where the cycles do not meet in the generic fiber in the global situation. See \cite{kr3}, Lemma 2.21.}, because otherwise there is some $i$ such that $x_i=0$ and then $\mathcal{Z}(x_i)=\mathcal{N}$ does not give any contribution.

In the following, we write $r(x_i):=v(h(x_i,x_i))$ and refer to this number (which, by our assumptions, is a nonnegative integer) as the ``valuation of $x_i$''.
\end{fliess}

In the remaining part of this section, we show Prop.\ \ref{dimprime}, i.e.\ that $\mathcal{Z}(\underline{x})_\textnormal{red}$ is a union of closed Bruhat-Tits strata $\mathcal{N}_\Lambda$ and that it is pure of some dimension only depending on the number of $x_i$ of valuation 0 involved.

\begin{lem} \label{ausbr}
Let $\underline{x}=(x_1,\dots,x_m)\in\mathbb{V}^m$. Let $\Lambda\in\mathcal{L}$ be a vertex such that the set $\mathcal{V}(\Lambda)$ of $\mathbb{F}$-valued points of the closed Bruhat-Tits stratum $\mathcal{N}_\Lambda$ intersects $\mathcal{Z}(\underline{x})(\mathbb{F})$ non-trivially. Then $\mathcal{V}(\Lambda)\subseteq\mathcal{Z}(p\underline{x})(\mathbb{F})$.
\end{lem}

\begin{proof}
Let $\Lambda\in\mathcal{L}$ be as assumed. Let $A\in\mathcal{V}(\Lambda)\cap\mathcal{Z}(\underline{x})(\mathbb{F})$. By definition, we have $x_i\in A\subseteq\Lambda\subseteq p^{-1}\Lambda^\sharp$ and thus $px_i\in\Lambda^\sharp$. But as $A\in\mathcal{V}(\Lambda)$ implies $\Lambda^\sharp\subseteq A^\sharp$, the condition $px_i\in\Lambda^\sharp$ is sufficient for $\mathcal{V}(\Lambda)\subseteq\mathcal{Z}(p\underline{x})$.
\end{proof}

\begin{fliess}
Let $\underline{x}\in\mathbb{V}^m$. We fix the following notation:
\[ \mathcal{S}(\underline{x}):=\{\Lambda\in\mathcal{L}\ |\ \mathcal{V}(\Lambda)\subseteq\mathcal{Z}(x)(\mathbb{F})\}. \]
$\mathcal{S}(\underline{x})$ will be viewed as a simplicial subcomplex of $\mathcal{L}$. Note that
\[ \Lambda\in\mathcal{S}(\underline{x})\Leftrightarrow x_i\in\Lambda^\sharp\ \forall i. \]
We will later also use the notations $\mathcal{S}(\underline{x})^{(t)}:=\mathcal{S}(\underline{x})\cap\mathcal{L}^{(t)}$ and $\mathcal{S}(\underline{x})^{\max}:=\mathcal{S}(\underline{x})^{(t_{\max})}$.
\end{fliess}

\begin{prop} \label{bts}
Let $\underline{x}\in\mathbb{V}^m$. We have the equality
\[ \mathcal{Z}(\underline{x})_\textnormal{red}=\bigcup_{\Lambda\in\mathcal{S}(\underline{x})}\mathcal{N}_\Lambda. \]
In other words, $\mathcal{Z}(\underline{x})_\textnormal{red}$ is a union of Bruhat-Tits strata.
\end{prop}

\begin{remark} \label{bts2}
Using the properties of the Bruhat-Tits stratification, the proposition states that the simplicial complex $\mathcal{S}(\underline{x})$ contains the complete information about the global structure of $\mathcal{Z}(\underline{x})_\textnormal{red}$. In particular,
\[ \mathcal{Z}(\underline{x})_\textnormal{red}\ \textnormal{is connected if and only if}\ \mathcal{S}(\underline{x})\ \textnormal{is}. \]
\end{remark}

\begin{proof}[Proof of \ref{bts}]
As $\mathcal{Z}(\underline{x})_\textnormal{red}$ is a closed subscheme of $\mathcal{N}_\textnormal{red}$, hence locally of finite type over the algebraically closed field $\mathbb{F}$, it is enough to show
\[ \mathcal{Z}(\underline{x})(\mathbb{F})=\bigcup_{\Lambda\in\mathcal{S}(\underline{x})}\mathcal{V}(\Lambda). \]
For $m=n$, this is Prop.\ 4.1 of \cite{kr}.\\
For $m<n$, fix $\underline{z}=(z_{m+1},\dots,z_n)\in\mathbb{V}^{n-m}$ such that $(x_1,\dots,x_m,z_{m+1},\dots,z_n)$ forms an orthogonal basis of $\mathbb{V}$. Let  $A\in\mathcal{Z}(\underline{x})(\mathbb{F})$ be arbitrary. Let $\Lambda\in\mathcal{L}$ be of minimal type such that $A\in\mathcal{V}(\Lambda)=\mathcal{N}_\Lambda(\mathbb{F})$. The minimal type condition means $A\in\mathcal{N}_\Lambda^\circ(\mathbb{F})$. We have to show that the whole set $\mathcal{V}(\Lambda)$ belongs to $\mathcal{Z}(\underline{x})(\mathbb{F})$.

Since $\mathcal{N}_\textnormal{red}$ is connected and the irreducible components are the Bruhat-Tits strata coming from vertices of maximal type, Lemma \ref{ausbr} implies that for any special homomorphism $y$ and any vertex $\tilde\Lambda$, we find an integer $l_0$ s.t.
\[ \mathcal{V}(\tilde\Lambda)\subseteq\mathcal{Z}(p^ly)(\mathbb{F})\ \textnormal{for}\ l\geq l_0. \]
In particular, $\mathcal{V}(\Lambda)\subseteq\mathcal{Z}(p^l\underline{z})(\mathbb{F})$ for $l\gg 0$. Thus $A\in\mathcal{Z}(\underline{x},p^l\underline{z})$. Prop.\ 4.1 of \cite{kr} now implies $\mathcal{V}(\Lambda)\subseteq\mathcal{Z}(\underline{x},p^l\underline{z})\subseteq\mathcal{Z}(\underline{x})$.
\end{proof}

\begin{prop} \label{dimension}
For $\underline{x}\in\mathbb{V}^m$, let
\[ m_1:=\textnormal{Card}\{1\leq i\leq m\ |\ r(x_i)=0\}. \]
Then $\mathcal{Z}(\underline{x})_\textnormal{red}$ is of pure dimension
\[ \dim\,\mathcal{Z}(\underline{x})_\textnormal{red}=\lfloor(n-m_1-1)/2\rfloor, \]
in other words, of pure codimension $\lfloor m_1/2\rfloor$ in $\mathcal{N}_\textnormal{red}$ if $n$ is even, of pure codimension $\lfloor (m_1+1)/2\rfloor$ in $\mathcal{N}_\textnormal{red}$ if $n$ is odd.
\end{prop}

\begin{proof}
For $m=n$, this is Cor.\ 4.3 of \cite{kr}.\\
For $m<n$, choose $\underline{z}$ with the same properties as in the proof of Prop.\ \ref{bts}. By Prop.\ \ref{bts}, all irreducible components of $\mathcal{Z}(\underline{x})_\textnormal{red}$ are Bruhat-Tits strata. Let $\Lambda,\tilde\Lambda\in\mathcal{S}(\underline{x})$ be such that $\mathcal{N}_\Lambda,\mathcal{N}_{\tilde\Lambda}$ are irreducible components of $\mathcal{Z}(\underline{x})_\textnormal{red}$. We have to show that $t(\Lambda)=t(\tilde\Lambda)$.\\
But by the same reasoning as in the mentioned proof, both $\mathcal{V}(\Lambda)$ and $\mathcal{V}(\tilde\Lambda)$ are in $\mathcal{Z}(p^l\underline{z})(\mathbb{F})$ for $l\gg 0$. Thus, they are sets of $\mathbb{F}$-valued points of irreducible components of $\mathcal{Z}(\underline{x},p^l\underline{z})(\mathbb{F})$. Now apply Cor.\ 4.3 of \cite{kr}.
\end{proof}

\section{The Kudla-Rapoport stratification and related concepts}

For the next four sections, we assume $m=1$. In this section, we introduce the lattice-theoretical stratification of $\mathcal{Z}(x)(\mathbb{F})$, following Kudla and Rapoport's unpublished notes \cite{kr2}. For convenience of the reader, we give complete proofs. We also generalize the ideas of \cite{kr2} to give an analogous ``stratification'' of the set $\mathcal{L}$ of vertices.

\begin{fliess}
Let $x\in C\subset N_0$. Write $x_0:=x$, $x_1:=F^{-1}x\in N_1$. Set $r=v(h(x,x))$. We introduce the following notation:
\[ N_\parallel:=W_\mathbb{Q}x_0+W_\mathbb{Q}x_1\subseteq N,\quad N_\perp:=N_\parallel^\perp, \]
where $^\perp$ denotes the orthogonal complement w.r.t.\ $\langle .,.\rangle$. We then have orthogonal projections
\[ \textnormal{pr}_\parallel: N\twoheadrightarrow N_\parallel,\quad \textnormal{pr}_\perp: N\twoheadrightarrow N_\perp \]
Both $N_\parallel$ and $N_\perp$ are $F$-, $V$-, $\mathbb{Z}_{p^2}$-stable and $\langle .,.\rangle$ induces a non-degenerate form on both of them.

For $i=0,1$, write
\[ N_{\parallel,i}:=N_i\cap N_\parallel,\quad N_{\perp,i}:=N_i\cap N_\perp. \]
As $x\in C$, the $W_\mathbb{Q}$-vector spaces $N_{\parallel,0}$ and $N_{\perp,0}$ have a canonical $\mathbb{Q}_{p^2}$-rational structure induced by $C$. In other words, writing 
\[ C_\parallel:=N_{\parallel,0}\cap C,\quad C_\perp:=N_{\perp,0}\cap C, \]
we have 
\[ N_{\parallel,0}=(C_\parallel)\otimes_{\mathbb{Q}_{p^2}} W_\mathbb{Q},\quad N_{\perp,0}=(C_\perp)\otimes_{\mathbb{Q}_{p^2}} W_\mathbb{Q}. \]
In particular, $C_\perp$ is a $(n-1)$-dimensional $\mathbb{Q}_{p^2}$-vector space. The hermitian form $h$ induces a non-degenerate hermitian form, which we also denote by $h$, on $C_\perp$.

For any $W$-lattice $M$ in $N$, write
\[ M_\parallel:=M\cap N_\parallel,\quad M_\perp:=M\cap N_\perp. \]
If $M$ is $\mathbb{Z}_{p^2}$-stable, we also write for $i=0,1$
\[ M_{\parallel,i}=M\cap N_{\parallel,i},\quad M_{\perp,i}=M\cap N_{\perp,i}. \]
Furthermore, for any $W$-lattice $L$ in $N_\parallel$ (resp.\ $N_\perp$), we will denote the dual lattice of $L$ w.r.t.\ the form induced by $\langle .,.\rangle$ on $N_\parallel$ (resp.\ $N_\perp$) by $L^\vee$.

Let $M$ be a $W$-lattice in $N$. One has obvious inclusions
\[ M_\parallel\oplus M_\perp\subseteq M\subseteq\textnormal{pr}_\parallel(M)\oplus\textnormal{pr}_\perp(M). \]
\end{fliess}

\begin{lem}[Kudla-Rapoport, \cite{kr2}, Lemma 5.1] \label{fortytwo}
Let $M$ be self-dual. Then
\[ M_\parallel^\vee=\textnormal{pr}_\parallel(M),\ M_\perp^\vee=\textnormal{pr}_\perp(M). \]
The quotient module
\[ M/(M_\parallel\oplus M_\perp)\subseteq (M_\parallel^\vee\oplus M_\perp^\vee)/(M_\parallel\oplus M_\perp)=(M_\parallel^\vee/M_\parallel)\oplus(M_\perp^\vee/M_\perp) \]
is the graph of a $W$-linear isomorphism
\[ \gamma:M_\parallel^\vee/M_\parallel\rightarrow M_\perp^\vee/M_\perp, \]
which is an anti-isometry w.r.t.\ the $W_\mathbb{Q}/W$-valued symplectic forms induced by $\langle .,.\rangle$ on $M_\parallel^\vee/M_\parallel$ and $M_\perp^\vee/M_\perp$.

If $M$ is $F$-, $V$-, and $\mathbb{Z}_{p^2}$-stable, we have induced $\sigma$-linear endomorphisms $\overline{F}$ and $\sigma^{-1}$-linear endomorphisms $\overline{V}$, as well as induced $\mathbb{Z}_{p^2}$-actions, on $M_\parallel^\vee/M_\parallel$ and $M_\perp^\vee/M_\perp$. In this case, $\gamma$ is equi\-variant for them.
\end{lem}

\begin{proof}
Let $z\in\textnormal{pr}_\parallel(M)$. By definition, there exists $y\in N_\perp$ with $z+y\in M$. For any $z'\in M_\parallel$, by self-duality of $M$, we know that $\langle z,z'\rangle=\langle z+y,z'\rangle$ is in $W$. Thus, $z\in M_\parallel^\vee$.

By duality, the other inclusion $M_\parallel^\vee\subseteq\textnormal{pr}_\parallel(M)$ is equivalent to the following claim: $\textnormal{pr}_\parallel(M)^\vee\subseteq M_\parallel$.

Let $z\in\textnormal{pr}_\parallel(M)^\vee$. Let $z'+y'\in M$, with $z'\in N_\parallel$ and $y'\in N_\perp$. Then $\langle z,z'+y'\rangle=\langle z,z'\rangle$ is in $W$, since $z'\in\textnormal{pr}_\parallel(M)$. This shows $z\in M^\vee$. By self-duality, we have $z\in M$, proving the claim and thus the first equality
\[ M_\parallel^\vee=\textnormal{pr}_\parallel(M). \]
Next, we show that the projection
\[ \overline{\textnormal{pr}_\parallel}:M/(M_\parallel\oplus M_\perp)\rightarrow\textnormal{pr}_\parallel(M)/M_\parallel \]
is a $W$-module isomorphism. We have the following commutative diagram:
\[ \xymatrix @C=0.5cm { M \ar@{->>}[r]^{\textnormal{pr}_\parallel} \ar[d] & {\textnormal{pr}_\parallel(M)} \ar@{->>}[d] \\
{M/(M_\parallel\oplus M_\perp)} \ar[r]^-{\overline{\textnormal{pr}_\parallel}} & {\textnormal{pr}_\parallel(M)/M_\parallel} } \]
From this, the surjectivity of $\overline{\textnormal{pr}_\parallel}$ is immediate. To show the injectivity, take any $\overline{z}\in M/(M_\parallel\oplus M_\perp)$ with $\overline{pr}_\parallel(\overline{z})=0$. Let $z$ be a preimage of $\overline{z}$ in $M$. We have to show $z\in M_\parallel\oplus M_\perp$. But by choice of $z$, we know that $\textnormal{pr}_\parallel(z)$ is in $M_\parallel$, and thus $z-\textnormal{pr}_\parallel z\in M\cap N_\perp=M_\perp$.

Of course, the same proofs (with the $\parallel$s and the $\perp$s interchanged) work for the other assertions so far.

All this means that $M/(M_\parallel\oplus M_\perp)$ is the graph of the $W$-linear isomorphism
\[ \gamma:=\overline{\textnormal{pr}_\perp}\circ(\overline{\textnormal{pr}_\parallel})^{-1}:M_\parallel^\vee/M_\parallel\rightarrow M_\perp^\vee/M_\perp. \]
Explicitly, if $\overline{z}\in M_\parallel^\vee/M_\parallel$ is the coset of $z\in M_\parallel^\vee$, then $\gamma(\overline{z})\in M_\perp^\vee/M_\perp$ is the coset of elements $y\in M_\perp^\vee$ such that $z+y\in M$.

The symplectic form $\langle .,.\rangle$ is $W$-valued on $M_\parallel^\vee\times M_\parallel$ and $M_\perp^\vee\times M_\perp$ and thus induces $W_\mathbb{Q}/W$-valued forms $\overline{\langle .,.\rangle}$ on $M_\parallel^\vee/M_\parallel$ and $M_\perp^\vee/M_\perp$. Denote the cosets of $z$, $z'\in M_\parallel^\vee$ by $\overline{z}$ resp.\ $\overline{z}'\in M_\parallel^\vee/M_\parallel$. Let $y$, $y'\in M_\perp^\vee$ be representatives of $\gamma(\overline{z})$ resp.\ $\gamma(\overline{z}')$. Then we have 
\[ \langle z,z'\rangle+\langle y,y'\rangle=\langle z+y,z'+y'\rangle\in W, \]
since $z+y$ and $z'+y'$ are in $M=M^\vee$, and thus
\[ \overline{\langle\overline{z},\overline{z}'\rangle}+\overline{\langle\gamma(\overline{z}),\gamma(\overline{z}')\rangle}=0. \]
In other words, $\gamma$ is an anti-isometry for $\overline{\langle .,.\rangle}$.

Now assume $M$ to be $F$-stable. Let $\overline{z}\in M_\parallel^\vee/M_\parallel$. We want to show the $\overline{F}$-equivariance of $\gamma$, i.e.
\[ \gamma(\overline{F}\overline{z})=\overline{F}\gamma(\overline{z}). \]
Let $z\in M_\parallel^\vee$ be a representative of $\overline{z}$ and $y\in M_\perp^\vee$ be a representative of $\gamma(\overline{z})$, i.e.\ $z+y\in M$. Then $Fz$ is a representative of $\overline{F}\overline{z}$ and $Fy$ be a representative of $\overline{F}\gamma(\overline{z})$ by definition of $\overline{F}$. But, by $F$-invariance of $M$, we have
\[ Fz+Fy=F(z+y)\in M.\]
This means that $Fy$ is a representative of $\gamma(\overline{F}\overline{z})$, i.e.\ $\overline{F}\gamma(\overline{z})=\gamma(\overline{F}\overline{z})$.

The proofs of $\overline{V}$- and $\mathbb{Z}_{p^2}$-equivariance work analogously.
\end{proof}

\begin{remark} \label{thirtysix}
One trivial, but particularly important special case is the equivalence
\[ M_\parallel=M_\parallel^\vee\Leftrightarrow M_\perp=M_\perp^\vee\Leftrightarrow M=M_\parallel\oplus M_\perp. \]
\end{remark}

\begin{fliess}[Kudla-Rapoport, \cite{kr2}, \S5]
For any pair of integers $(a,b)$, write
\[ L_{(a,b)}:=p^{-a}Wx_0+p^{-b}Wx_1\subset N_\parallel. \]
Any $M\in\mathcal{N}(\mathbb{F})$ decomposes into eigenspaces for the $\mathbb{Z}_{p^2}$-action, thus so does $M_\parallel$, hence $M_\parallel$ is of the form $M_\parallel=L_{(a,b)}$ for some pair of integers $(a,b)$. Since
\[ FL_{(a,b)}=p^{-b}Wx_0+p^{-a+1}Wx_1=L_{(b,a-1)}, \]
the signature condition for $M$ implies $a-1\leq b\leq a$. An easy calculation yields 
\[ L_{(a,b)}^\vee=L_{(r-b,r-a)}, \]
where $r=v(h(x,x))$ as defined at the beginning of this section. Thus the inclusion $M_\parallel\subseteq M_\parallel^\vee$ implies $a+b\leq r$.
\end{fliess}

\begin{fliess}
A given $M\in\mathcal{N}(\mathbb{F})$ is an $\mathbb{F}$-valued point of the special cycle $\mathcal{Z}(x)$ if and only if $x_0=x\in M_0^\sharp=FM_1$, i.e.\ iff $x_1\in M_1$. Writing $M_\parallel=L_{(a,b)}$, this holds if and only if $a,b\geq 0$. We write
\[ \mathcal{Z}(x)^{(a,b)}:=\{M\in\mathcal{Z}(x)(\mathbb{F})\ |\ M_\parallel=L_{(a,b)}\}. \]
This yields a stratification (which will be referred to as the \itshape Kudla-Rapoport stratification\normalfont)\footnote{One may ask whether the Kudla-Rapoport strata are ``strata'' in the geometric sense, i.e.\ arise as sets of $\mathbb{F}$-valued points of locally closed subsets of $\mathcal{Z}(x)$. This is probably true, but we will not use this.}
\[ \mathcal{Z}(x)(\mathbb{F})=\mathcal{Z}(x)^{(0,0)}\cup\mathcal{Z}(x)^{(1,0)}\cup\mathcal{Z}(x)^{(1,1)}\cup\mathcal{Z}(x)^{(2,1)}\cup\dots\cup\mathcal{Z}(x)^{max}, \]
where $\mathcal{Z}(x)^{max}$ denotes the last occurring stratum, i.e.\ $\mathcal{Z}(x)^{\left(\frac{r}{2},\frac{r}{2}\right)}$ if $r$ is even, $\mathcal{Z}(x)^{\left(\frac{r+1}{2},\frac{r-1}{2}\right)}$ if $r$ is odd. (Of course, the union is disjoint.)
\end{fliess}

\begin{fliess} \label{lambda}
We may also consider vertices $\Lambda$ as self-dual lattices in $N$. To do this, one simply inverts the procedure of Prop.\ \ref{nnull}. Given a vertex $\Lambda\in\mathcal{L}$, i.e.\ a $\mathbb{Z}_{p^2}$-lattice in $C$, we associate the $\tau$-invariant $W$-lattice $\Lambda_W:=\Lambda\otimes_{\mathbb{Z}_{p^2}}W$ in $N_0$. Then we set
\[ \Lambda_0:=\Lambda_W,\quad \Lambda_1=F^{-1}\Lambda_W^\sharp \]
The $(\mathbb{Z}/2\mathbb{Z})$-graded $W$-module $\Lambda_\bullet:=\Lambda_0\oplus\Lambda_1$ is a $W$-lattice in $N$ which is obviously $F$-, $V$- and $\mathbb{Z}_{p^2}$-invariant, self-dual (w.r.t.\ $\langle .,.\rangle$) by construction, and satisfies the signature condition
\[ p\Lambda_0\subseteq F\Lambda_1=V\Lambda_1\overset{\text{\tiny\itshape t}}{\subset}\Lambda_0,\quad p\Lambda_1\overset{\text{\tiny\itshape t}}{\subset} F\Lambda_0=V\Lambda_0\subseteq\Lambda_1, \]
where $t=t(\Lambda)$. We write 
\[ \Lambda_\parallel:=(\Lambda_\bullet)_\parallel,\quad \Lambda_\perp:=(\Lambda_\bullet)_\perp \]
and observe that $\Lambda_\bullet$ satisfies the assumptions of Lemma \ref{fortytwo}. Writing $\Lambda_\parallel=L_{(a,b)}$, we carry through the computations of the last two paragraphs. They yield the analogous results
\[ a-1\leq b\leq a;\quad a+b\leq r;\quad \Lambda\in \mathcal{S}(x)\Leftrightarrow x_1\in\Lambda_1\Leftrightarrow b\geq 0. \]
Thus we have a ``stratification'' of the simplicial complex $\mathcal{S}(x)$ analogous to the Kudla-Rapoport stratification on $\mathcal{Z}(x)(\mathbb{F})$, i.e.\ we can write $\mathcal{S}(x)$ as a disjoint union
\[ \mathcal{S}(x)=\mathcal{S}(x)^{(0,0)}\cup\mathcal{S}(x)^{(1,0)}\cup\dots\cup\mathcal{S}(x)^{\max}, \]
where $\mathcal{S}(x)^{\max}=\mathcal{S}(x)^{\left(\frac{r}{2},\frac{r}{2}\right)}$ if $r$ is even and $\mathcal{S}(x)^{\max}=\mathcal{S}(x)^{\left(\frac{r+1}{2},\frac{r-1}{2}\right)}$ if $r$ is odd.
\end{fliess}

\begin{fliess}
Finally, we translate our stratification results into the $(N_0,h)$ language. Recall from the proof of Prop.\ \ref{nnull} that we have a bijection
\begin{align*}
\{ F\textnormal{-, }V\textnormal{-, }\mathbb{Z}_{p^2}\textnormal{-invariant self-dual } W\textnormal{-lattices in } N\} &\leftrightarrow \{ W\textnormal{-lattices in } N_0 \} \\
M=M_0\oplus M_1 &\mapsto M_0
\end{align*}
with inverse
\[ A\mapsto A\oplus F^{-1}A^\sharp, \]
and that under this identification, the chain condition in $N$:
\[ pM_0\overset{\text{\tiny{\itshape n-t}}}{\subset} FM_1\overset{\text{\tiny\itshape t}}{\subset} M_0,\quad pM_1\overset{\text{\tiny\itshape t}}{\subset} FM_0\overset{\text{\tiny{\itshape n-t}}}{\subset} M_1 \]
is equivalent to the chain condition in $N_0$:
\[ pM_0\overset{\text{\tiny\itshape{n-t}}}{\subset} M_0^\sharp\overset{\text{\tiny\itshape t}}{\subset} M_0. \]
\end{fliess}

\begin{lem} \label{leo}
The assertion $M_\parallel=L_{(a,b)}$, i.e.
\[ M_\parallel=p^{-a}Wx_0\oplus p^{-b}Wx_1 \]
is equivalent to
\[ M_{\parallel,0}=p^{-a}Wx,\quad (M_0^\sharp)\cap N_{\parallel,0}=p^{-b}Wx, \]
where $x=x_0$.
\end{lem}

\begin{proof}
The first identity is trivial. For the second, put together $(M_0^\sharp)\cap N_{\parallel,0}=(FM_1)\cap N_{\parallel,0}$ and $F(p^{-b}Wx_1)=p^{-b}WFx_1=p^{-b}Wx_0$.
\end{proof}

\section{Base case I: The case $r=0$}

We have now introduced all the tools needed for our proof of Thms.\ \ref{prime} and \ref{holygrail}. The aim of this section is to prove Thm.\ \ref{prime}(1), which corresponds to $r=r(x)=v(h(x,x))=0$, which we assume throughout this section. We will also prove Thm.\ \ref{holygrail} in this case. The results are from the unpublished notes \cite{kr2}; we give an elaboration of those results and the proofs.

\begin{fliess}
We saw in Proposition \ref{nnull} that the set $\mathcal{N}(\mathbb{F})$ may be reconstructed from the knowledge of $(N_0,h)$ alone. Since $r$ is even, $h$ is odd on $C_\perp$ (i.e.\ has odd order of determinant), which means that $(N_{\perp,0},h)$ is isomorphic to the $\varphi_0$-eigenspace of the isocrystal obtained from the data used in defining the moduli scheme $\mathcal{N}(1,n-2)$. We fix an isomorphism, which then induces a bijection
\[ \mathcal{N}(1,n-2)(\mathbb{F})\leftrightarrow\{A_\perp\subset N_{\perp,0}\ \textnormal{a}\ W\textnormal{-lattice}\ |\ pA_\perp\overset{\text{\tiny 1}}{\subset}A_\perp^\sharp\overset{\text{\tiny{\itshape n-\normalfont 2}}}{\subset}A_\perp\}. \]
The right hand side will be denoted by $\mathcal{N}_\perp$. Of course, $\mathcal{N}_\perp$ can be identified with the set of self-dual $F$-, $V$- and $\mathbb{Z}_{p^2}$-invariant $W$-lattices $M_\perp$ in $N_\perp$ satisfying the signature condition
\[ pM_{\perp,0}\overset{\text{\tiny{n-2}}}{\subset} FM_{\perp,1}\overset{\text{\tiny 1}}{\subset}M_{\perp,0},\quad pM_{\perp,1}\overset{\text{\tiny 1}}{\subset} FM_{\perp,0}\overset{\text{\tiny{n-2}}}{\subset}M_{\perp} \]
by the usual construction.
\end{fliess}

\begin{prop}[Kudla-Rapoport, \cite{kr2}, Prop.\ 5.2]\label{zero}
One has a map
\[ \mathcal{N}_\perp\rightarrow\mathcal{Z}(x)(\mathbb{F}),\quad M_\perp\mapsto L_{(0,0)}\oplus M_\perp\subset N, \]
which is a bijection with inverse
\[ \mathcal{Z}(x)(\mathbb{F})\rightarrow\mathcal{N}_\perp,\quad M\mapsto M_\perp. \]
\end{prop}

\begin{proof}
Let $M_\perp\in\mathcal{N}_\perp$ be given. Using Remark \ref{thirtysix}, it follows immediately from the self-duality of $M_\perp$ that $L_{(0,0)}\oplus M_\perp$ is self-dual. The stability conditions for $L_{(0,0)}\oplus M_\perp$ follows from the ones for $M_\perp$ and $L_{(0,0)}$, and the signature condition follows from the one for $M_\perp$ using $FL_{(0,0)}=L_{(0,-1)}$. Finally, $x_1\in(L_{(0,0)}\oplus M_\perp)_1$ holds by construction. Thus $L_{(0,0)}\oplus M_\perp\in\mathcal{Z}(x)(\mathbb{F})$.\\
On the other hand, since $r=0$, the Kudla-Rapoport stratification is trivial, i.e.\ $\mathcal{Z}(x)=\mathcal{Z}(x)^{(0,0)}$. In other words, for $M\in\mathcal{Z}(x)(\mathbb{F})$, one has \[ M_\parallel=L_{(0,0)}=L_{(0,0)}^\vee=M_\parallel^\vee. \]
Thus, Lemma \ref{thirtysix} states that 
\begin{equation} \label{unnoetig} M=L_{(0,0)}\oplus M_\perp \end{equation}
and that $M_\perp$ is self-dual. The stability conditions for $M_\perp$ are trivial, and the signature condition follows from the one for $M$ and the equality $FL_{(0,0)}=L_{(0,-1)}$. Also, Eqn.\ \ref{unnoetig} shows that the two maps are inverse to each other.
\end{proof}

\begin{fliess}\label{nullbtg}
By the very same reasoning, using the ``translation lemma'' \ref{leo}, we get an identification of the set of vertices of type $t$ in $\mathcal{S}(x)$ with the set of $\mathbb{Z}_{p^2}$-lattices $\Lambda_\perp$ in $C_\perp$ satisfying the chain condition
\[ p\Lambda_\perp\subseteq\Lambda_\perp^\sharp\overset{\text{\tiny\itshape t}}{\subset}\Lambda_\perp. \]
Thus, applying Theorem 3.6 of \cite{voll}, we have an isomorphism of simplicial complexes 
\[ \mathcal{S}(x)\cong\mathfrak{B}(SU(C_\perp,ph),\mathbb{Q}_p), \]
proving Thm.\ \ref{prime}(1).
\end{fliess}

\begin{cor}
Theorem \ref{holygrail} holds true for $\underline{x}=(x)$, i.e.\ $\mathcal{Z}(x)_\textnormal{red}$ is connected.
\end{cor}

\begin{proof}
It is a general fact from Bruhat-Tits theory (see e.g.\ \cite{ti}, 2.2.1) that the simplicial complex of the Bruhat-Tits building $\mathfrak{B}(SU(C_\perp,ph),\mathbb{Q}_p)$ is connected. Now use the isomorphism of Thm.\ \ref{prime}(1) and Remark \ref{bts2}, which states that connectedness of $\mathcal{S}(x)$ implies connectedness of $\mathcal{Z}(x)_\textnormal{red}$.
\end{proof}

\begin{remark}
One can show that the map of Proposition \ref{zero} defines an isomorphism of schemes $\mathcal{Z}(x)_\textnormal{red}\cong\mathcal{N}(1,n-2)_\textnormal{red}$. Indeed, Kudla and Rapoport show that $\mathcal{Z}(x)$ is isomorphic to $\mathcal{N}(1,n-2)$ as formal schemes. The proof roughly goes as follows:

W.l.o.g.\ assume $h(x,x)=1$. Write $e$ for the idempotent element 
\[ (1_\mathbb{X}-x\circ\lambda_{\overline{\mathbb{Y}}}^{-1}\circ x^\vee\circ\lambda_\mathbb{X})\in\textnormal{End}_{\mathbb{Z}_{p^2}}(\mathbb{X}). \]
Write $\dot{\mathbb{X}}$ for the $p$-divisible subgroup $e\mathbb{X}$. One shows that $\lambda_\mathbb{X}|_{\dot{\mathbb{X}}}$ defines a $p$-principal polarization of $\dot{\mathbb{X}}$ and that $(\dot{\mathbb{X}},\iota_\mathbb{X}|_{\dot{\mathbb{X}}},\lambda_\mathbb{X}|_{\dot{\mathbb{X}}})$ is isomorphic to the basic object used in the definition of $\mathcal{N}(1,n-2)$. Now if $(X,\iota_X,\lambda_X,\rho_X)\in\mathcal{N}(S)$ is in $\mathcal{Z}(x)(S)$, then $e$ lifts to an endomorphism of $X$, which we also denote by $e$. Write $\dot{X}=eX$. Then one checks that restricting $\iota_X$, $\lambda_X$ and $\rho_X$ defines
\[ (\dot{X},\iota_{\dot{X}},\lambda_{\dot{X}},\rho_{\dot{X}})\in\mathcal{N}(1,n-2)(S) \]
and that this construction is invertible.\\
We will not need to use those results in the following.
\end{remark}

\section{Base case II: The case $r=1$}

We now come to the proof of Thm.\ \ref{prime}(2), i.e.\ to the case $r=r(x)=v(h(x,x))=1$, which we assume throughout this section. We will also prove Thm.\ \ref{holygrail} and compute the Kudla-Rapoport stratification in this case.

\begin{fliess}
As in the last section, we will relate the global structure of $\mathcal{Z}(x)(\mathbb{F})$ to the simplicial complex of the building $\mathfrak{B}(SU(C_\perp,ph),\mathbb{Q}_p)$. We know from \ref{lambda} and \cite{voll}, Thm.\ 3.6, that we have an isomorphism of simplicial complexes
\[ \mathcal{L}_\perp\cong\mathfrak{B}(SU(C_\perp,ph),\mathbb{Q}_p), \]
where the underlying set of $\mathcal{L}_\perp$ is the set of all $F$-, $V$- and $\mathbb{Z}_{p^2}$-invariant $W$-lattices $\Lambda_\perp$ in $N_\perp$ which are self-dual w.r.t.\ $\langle .,.\rangle$ and which are $\tau$-invariant and satisfy some chain condition
\[ p\Lambda_{\perp,0}\subseteq\Lambda_{\perp,0}^\sharp\subseteq\Lambda_{\perp,0}, \]
and where two distinct lattices $\Lambda_\perp$, $\tilde\Lambda_\perp$ in $\mathcal{L}_\perp$ by definition neighbour each other if and only if $\Lambda_{\perp,0}\subset\tilde\Lambda_{\perp,0}$ or $\tilde\Lambda_{\perp,0}\subset\Lambda_{\perp,0}$. We will use the terms ``vertex'', ``type'', $\mathcal{L}_\perp^{(t)}$ etc. as in the ``non-$\perp$'' situation.

Note that, since $r$ is odd, the hermitian form $h$ on $C_\perp$ has even order of determinant. Thus, the type of a vertex in $\mathfrak{B}(SU(C_\perp,ph),\mathbb{Q}_p)$ is an \itshape even \normalfont integer between 0 and $n-1$.
\end{fliess}

\begin{remark}
By Prop.\ \ref{dimension}, the variety $\mathcal{Z}(x)_\textnormal{red}$ is of pure dimension $\dim\,\mathcal{N}_\textnormal{red}$. Its irreducible components are thus all of the form $\mathcal{N}_\Lambda$ with $\Lambda\in\mathcal{L}^{\max}$. Hence in order to show connectedness it is sufficient to show that all vertices of maximal type in $\mathcal{S}(x)$ are in the same connected component of $\mathcal{S}(x)$.
\end{remark}

\begin{lem} \label{einslambda}
For any $\Lambda\in \mathcal{S}(x)^{\max}$, we have $\Lambda_\parallel=L_{(1,0)}$.
\end{lem}

\begin{proof}
Let $\Lambda\in \mathcal{S}(x)^{\max}$. From the general considerations in \ref{lambda}, writing $\Lambda_\parallel=L_{(a,b)}$, we get either $(a,b)=(1,0)$ or $(a,b)=(0,0)$. We have to show that the latter case does not occur.

We first assume $n$ to be odd. Then $t_{\max}=n$ and thus the signature condition for $\Lambda$ may be written as
\[ p\Lambda_0=F\Lambda_1\overset{\text{\tiny\itshape n}}{\subset}\Lambda_0,\quad p\Lambda_1\overset{\text{\tiny\itshape n}}{\subset}F\Lambda_0=\Lambda_1. \]
In particular, we have
\[ L_{(b,a-1)}=F\Lambda_\parallel=p\Lambda_{\parallel,0}\oplus\Lambda_{\parallel,1}=L_{(a-1,b)}. \]
This means $b=a-1$, and if we ask for $\Lambda$ to be in $\mathcal{S}(x)$, we get $\Lambda_\parallel=L_{(1,0)}$.

Now assume $n$ to be even. Then $t_{\max}=n-1$ and a vertex $\Lambda$ of maximal type satisfies the signature condition 
\[ p\Lambda_0\overset{\text{\tiny 1}}{\subset}F\Lambda_1\overset{\text{\tiny{\itshape n-\normalfont 1}}}{\subset}\Lambda_0,\ p\Lambda_1\overset{\text{\tiny{\itshape n-\normalfont 1}}}{\subset}F\Lambda_0\overset{\text{\tiny 1}}{\subset}\Lambda_1. \]
Now unfortunately $F\Lambda_\parallel=p(\Lambda_0)_\parallel\oplus(\Lambda_1)_\parallel$ does not follow immediately, so we have to put more work into this case.

Assume $\Lambda_\parallel=L_{(0,0)}$. By Lemma \ref{fortytwo}, the $W$-module $\Lambda_\bullet/(\Lambda_\parallel\oplus\Lambda_\perp)$ is the graph of an isomorphism $\gamma:\Lambda_\parallel^\vee/\Lambda_\parallel\rightarrow\Lambda_\perp^\vee/\Lambda_\perp$. Since $\textnormal{pr}_\parallel$ and $\textnormal{pr}_\perp$ map $N_0$ to $N_{\parallel,0}=(N_\parallel)_0$ and $N_{\perp,0}$, respectively, this isomorphism restricts to an isomorphism
\[ \gamma_0:\Lambda_{\parallel,1}^\vee/\Lambda_{\parallel,0}\rightarrow\Lambda_{\perp,1}^\vee/\Lambda_{\perp,0} \]
whose graph is $\Lambda_0/(\Lambda_{\parallel,0}\oplus\Lambda_{\perp,0})$. Of course, an analogous result holds for the degree 1 component. We have inclusion diagrams
\[
\begin{matrix}
 {}&{}&\Lambda_{\parallel,1}^\vee\oplus\Lambda_{\perp,1}^\vee\\
 {}&{}&\cup\\
 F\Lambda_1&\overset{n-1}{\subset}&\Lambda_0\\
 \cup&{}&\cup\\
 F(\Lambda_{\parallel,1}\oplus\Lambda_{\perp,1})&\overset{\alpha}{\subset}&\Lambda_{\parallel,0}\oplus\Lambda_{\perp,0},
 \end{matrix}
 \qquad\qquad
 \begin{matrix}
 \Lambda_{\parallel,0}^\vee\oplus\Lambda_{\perp,0}^\vee{}&{}\\
 \cup&{}&{}\\
 \Lambda_1&\overset{1}{\supset}&F\Lambda_0\\
 \cup&{}&\cup\\
 \Lambda_{\parallel,1}\oplus\Lambda_{\perp,1}&\overset{\beta}{\supset}&F(\Lambda_{\parallel,0}\oplus\Lambda_{\perp,0}).
 \end{matrix}
\]
The indices for the inclusions in the middle columns all equal 1 by the graph condition just mentioned, since $\Lambda_{\parallel}^\vee/\Lambda_{\parallel}=\mathbb{F}\cdot\overline{p^{-1}x_0}\oplus\mathbb{F}\cdot\overline{p^{-1}x_1}$. The index for the left column equals 1 by applying $F$ on the third column, while the index for the right column equals 1 by applying $F$ on the second column. Finally, the indices in the middle row come from the signature condition for $\Lambda$. It follows that $\alpha=n-1$ and $\beta=1$, hence the signature condition for $\Lambda_\parallel\oplus\Lambda_\perp\subsetneq\Lambda_\bullet$:
\begin{align*}
p(\Lambda_{\parallel,0}\oplus\Lambda_{\perp,0})\overset{\text{\tiny 1}}{\subset}F(\Lambda_{\parallel,1}\oplus\Lambda_{\perp,1})\overset{\text{\tiny{\itshape n-\normalfont 1}}}{\subset}\Lambda_{\parallel,0}\oplus\Lambda_{\perp,0},\\
p(\Lambda_{\parallel,1}\oplus\Lambda_{\perp,1})\overset{\text{\tiny{\itshape n-\normalfont 1}}}{\subset}F(\Lambda_{\parallel,0}\oplus\Lambda_{\perp,0})\overset{\text{\tiny 1}}{\subset}\Lambda_{\parallel,1}\oplus\Lambda_{\perp,1}.
\end{align*}
By assumption, we have
\[ F\Lambda_\parallel=FL_{(0,0)}=L_{(0,-1)}=\Lambda_{\parallel,0}\oplus p\Lambda_{\parallel,1} \]
and thus
\begin{equation} \label{sig} p\Lambda_{\perp,0}=F\Lambda_{\perp,1}\overset{\text{\tiny{\itshape n-\normalfont 1}}}{\subset}\Lambda_{\perp,0},\quad p\Lambda_{\perp,1}\overset{\text{\tiny{\itshape n-\normalfont 1}}}{\subset}F\Lambda_{\perp,0}=\Lambda_{\perp,1}. \end{equation}
On the other hand,
\[ \Lambda_\parallel^\vee/\Lambda_\parallel=L_{(1,1)}/L_{(0,0)}=\mathbb{F}\cdot\overline{p^{-1}x_0}\oplus\mathbb{F}\cdot\overline{p^{-1}x_1} \]
is endowed with a $\sigma$-linear endomorphism $\overline{F}$ induced by the Frobenius $F$. As usual, we write $\overline{F}_0$ resp.\ $\overline{F}_1$ for the restriction of $\overline{F}$ to $(\Lambda_\parallel^\vee/\Lambda_\parallel)_0=\Lambda_{\parallel,1}^\vee/\Lambda_{\parallel,0}$ resp.\ $(\Lambda_\parallel^\vee/\Lambda_\parallel)_1$. Since $FL_{(1,1)}=L_{(1,0)}$, we get that $\overline{F}_1$ is an isomorphism, while $\overline{F}_0$ is the zero map.

We have the same notions of $\overline{F}_0$ and $\overline{F}_1$ on $\Lambda_\perp^\vee/\Lambda_\perp$. By Lemma \ref{fortytwo}, we have a $\mathbb{Z}_{p^2}$-linear isomorphism \[ \gamma:\Lambda_\parallel^\vee/\Lambda_\parallel\rightarrow\Lambda_\perp^\vee/\Lambda_\perp, \]
and $\gamma$ is $\overline{F}$-equivariant.

Thus, on $\Lambda_\perp^\vee/\Lambda_\perp$, we have that $\overline{F}_1$ is an isomorphism, while $\overline{F}_0$ is the zero map. The latter means $F\Lambda_{\perp,1}^\vee=\Lambda_{\perp,1}$. By the injectivity of $F$, we get that $F\Lambda_{\perp,0}$ is \itshape properly \normalfont contained in $\Lambda_{\perp}$, contradicting the signature condition \eqref{sig}.
\end{proof}

\begin{cor} \label{wurscht}
For any $\Lambda\in \mathcal{S}(x)^{\max}$, the corresponding self-dual $W$-lattice $\Lambda_\bullet$ in $N$ is of the form $\Lambda_\bullet=L_{(1,0)}\oplus\Lambda_\perp$ with
\[ \Lambda_\perp\in\mathcal{L}_\perp^{\max}=\mathcal{L}_\perp^{(t_{\max}-1)}. \]
\end{cor}

\begin{proof}
Lemma \ref{einslambda} states $\Lambda_\parallel=L_{(1,0)}$. This is self-dual, thus by Remark \ref{thirtysix}, we get the desired direct sum decomposition. Checking that $\Lambda_\perp$ is a vertex of maximal type $t_{\max}-1$ in $C_\perp$ is straightforward.
\end{proof}

\begin{fliess} \label{dooof}
We now come to the proof of Thm.\ \ref{prime}(2). It is immediate that, given a vertex $\Lambda_\perp\in\mathcal{L}_\perp$ in $C_\perp$, the lattice 
\[ \Lambda:=\left(L_{(1,0)}\oplus\Lambda_\perp\right)_0^{\tau=1} \]
is a vertex in $\mathcal{S}(x)$ of type $t(\Lambda)=t(\Lambda_\perp)+1$. Thus, we have injections of sets
\[ \Phi:\mathcal{L}_\perp^{(t)}\rightarrow \mathcal{S}(x)^{(t+1)},\quad \Lambda_\perp\mapsto \left(L_{(1,0)}\oplus\Lambda_\perp\right)_0^{\tau=1}. \]
We just proved that this map is a bijection in maximal type. However, it is never surjective in lower types. Since $\mathcal{Z}(x)_\textnormal{red}$ is of pure codimension 0, any vertex in $\mathcal{S}(x)$ is contained in some vertex in $\mathcal{S}(x)^{\max}$ and we can thus still compute $\mathcal{S}(x)$ from $\mathcal{L}_\perp$. We have
\[ \mathcal{S}(x)=\left\{ \Lambda\in\mathcal{L}\ \mid\ \exists\ \tilde\Lambda_\perp\in\mathcal{L}_\perp:\ \Lambda\subseteq\left(L_{(1,0)}\oplus\tilde\Lambda_\perp\right)_0^{\tau=1}\right\}. \]
Also, the induced map
\[ \Phi:\mathcal{L}_\perp\rightarrow \mathcal{S}(x) \]
is obviously a morphism of simplicial complexes. That is, if two vertices $\Lambda_\perp$, $\tilde\Lambda_\perp$ neighbour each other in $\mathfrak{B}(SU(C_\perp,ph),\mathbb{Q}_p)$, then $\Lambda=(L_{(1,0)}\oplus\Lambda_\perp)_0^{\tau=1}$ and $\tilde\Lambda=(L_{(1,0)}\oplus\tilde\Lambda_\perp)_0^{\tau=1}$ neighbour each other in $\mathcal{S}(x)$. This proves Thm.\ \ref{prime}(2).
\end{fliess}

\begin{remark}
Let $\Lambda\in\mathcal{L}$. Assume $\Lambda_\bullet=L_{(1,0)}\oplus\Lambda_\perp$. Then neighbours of $\Lambda_\perp$ give rise to neighbours of $\Lambda$ via $\Phi$. However, the vertex $\Lambda$ has more neighbours than $\Lambda_\perp$, and it may occur (not in maximal type) that two Bruhat-Tits strata intersect although the intersection of the corresponding vertices in $C_\perp$ (if those exist) is not a vertex in $C_\perp$.
\end{remark}

\begin{cor} \label{eins}
Theorem \ref{holygrail} holds true for $\underline{x}=(x)$, i.e.\ $\mathcal{Z}(x)_\textnormal{red}$ is connected.
\end{cor}

\begin{proof}
Consider two arbitrary irreducible components $\mathcal{N}_\Lambda$, $\mathcal{N}_{\tilde\Lambda}$ of $\mathcal{Z}(x)_\textnormal{red}$. As $\mathcal{Z}(x)_\textnormal{red}$ is of pure codimension 0 in $\mathcal{N}_\textnormal{red}$, those are irreducible components in $\mathcal{N}_\textnormal{red}$ and the corresponding vertices $\Lambda$ and $\tilde\Lambda$ are of maximal type. By the previous corollary, $\Lambda_\perp$ and $\tilde\Lambda_\perp$ are vertices of maximal type in $\mathcal{L}_\perp$.

As in the case $r=0$, the simplicial complex $\mathfrak{B}(SU(C_\perp,ph),\mathbb{Q}_p)$ is connected by Bruhat-Tits theory. By our identification of $\mathcal{L}_\perp$ with that complex, we find a ``path''
\[ \Lambda_\perp=\Lambda_\perp^{(0)},\Lambda_\perp^{(1)},\dots,\Lambda_\perp^{(s)}=\tilde\Lambda_\perp \]
consisting of vertices in $\mathcal{L}_\perp$ such that, for $1\leq i\leq s$, the two consecutive vertices $\Lambda_\perp^{(i-1)}$, $\Lambda_\perp^{(i)}$ neighbour each other. Now set
\[ \Lambda^{(i)}:=\left(L_{(1,0)}\oplus\Lambda_\perp^{(i)}\right)_0^{\tau=1}. \]
The $\Lambda^{(i)}$ are vertices in $C$, all in $\mathcal{S}(x)$, with $\Lambda^{(0)}=\Lambda$, $\Lambda^{(s)}=\tilde\Lambda$, and two consecutive $\Lambda^{(i)}$ neighbour each other.

Thus, $\Lambda$ and $\tilde\Lambda$ are in the same connected component of the simplicial complex $\mathcal{S}(x)$.
\end{proof}

We will now give an explicit computation of the Kudla-Rapoport stratification $\mathcal{Z}(x)(\mathbb{F})=\mathcal{Z}(x)^{(0,0)}\cup\mathcal{Z}(x)^{(1,0)}$.

\begin{lem}[Kudla-Rapoport, \cite{kr2}, Prop.\ 5.2]
All lattices $M$ in $\mathcal{Z}(x)^{(1,0)}$ have $\tau$-invariant degree zero component $M_0$.
\end{lem}

\begin{proof}
Let $M\in\mathcal{Z}(x)^{(1,0)}$. Then by self-duality of $L_{(1,0)}$ and Remark \ref{thirtysix}, we have $M=M_\parallel\oplus M_\perp=L_{(1,0)}\oplus M_\perp$. It is thus sufficient to show $\tau$-invariance of both $M_{\parallel,0}$ and $M_{\perp,0}$. The former is automatic, because $x\in C$. For the latter, we use the signature condition for $M$,
\begin{align*}
p(M_{\parallel,0}\oplus M_{\perp,0})\overset{\text{\tiny{\itshape n-\normalfont 1}}}{\subset}F(M_{\parallel,1}\oplus M_{\perp,1})\overset{\text{\tiny 1}}{\subset}M_{\parallel,0}\oplus M_{\perp,0},\\
p(M_{\parallel,1}\oplus M_{\perp,1})\overset{\text{\tiny 1}}{\subset}F(M_{\parallel,0}\oplus M_{\perp,0})\overset{\text{\tiny{\itshape n-\normalfont 1}}}{\subset}M_{\parallel,1}\oplus M_{\perp,1}.
\end{align*}
By assumption, we have
\[ FM_\parallel=FL_{(1,0)}=L_{(0,0)}=pM_{\parallel,0}\oplus M_{\parallel,1} \]
and thus
\[ pM_{\perp,0}\overset{\text{\tiny{\itshape n-\normalfont 1}}}{\subset}FM_{\perp,1}=M_{\perp,0},\quad pM_{\perp,1}=FM_{\perp,0}\overset{\text{\tiny{\itshape n-\normalfont 1}}}{\subset}M_{\perp,1}. \]
In particular, $F^2M_\perp=pM_\perp$, i.e.\ $\tau M_{\perp,0}=M_{\perp,0}$.
\end{proof}

The $\tau$-invariance of $M_0$ means that $M_0$ arises from some lattice in $C$ via scalar extension. Thus, all $M\in\mathcal{Z}(x)^{(1,0)}$ are of the form $M=\Lambda_\bullet=\Lambda_W\oplus F^{-1}\Lambda_W^\sharp$ for some type 1 vertex $\Lambda\in\mathcal{L}^{(1)}$. The converse does not hold in general, as will be clear from the examples $n=3,4$ below. We thus still have to give a necessary and sufficient condition for $\Lambda\in\mathcal{L}^{(1)}$ to define a lattice $\Lambda_\bullet$ in $\mathcal{Z}(x)^{(1,0)}$. Note that $\Lambda_\bullet\in\mathcal{Z}(x)^{(1,0)}$ means that $\Lambda_\bullet=L_{(1,0)}\oplus\Lambda_\perp$, i.e.\ $\Lambda=p^{-1}\mathbb{Z}_{p^2}x\oplus(\Lambda_\perp)_0^{\tau=1}$.

\begin{prop} \label{beliebigekodim}
Let $\Lambda\in\mathcal{S}(x)$ be of arbitrary type $t$. Then $\Lambda_\bullet=\Lambda_\parallel\oplus\Lambda_\perp$ if and only if all irreducible components of $\mathcal{N}_\textnormal{red}$ which contain $\mathcal{N}_\Lambda$ are contained in $\mathcal{Z}(x)_\textnormal{red}$.
\end{prop}

\begin{proof}
By \ref{muff}, the set of irreducible components $\mathcal{N}_{\tilde\Lambda}$ of $\mathcal{N}_\textnormal{red}$ which contain $\mathcal{N}_\Lambda$ is in bijection with the set of $((t^{\max}-t)/2)$-dimensional isotropic subspaces of the $(n-t)$-dimensional hermitian $\mathbb{F}_{p^2}$-vector space $(V',\overline{h})$, where $V':=\Lambda^\sharp/p\Lambda$ and $\overline{h}$ is induced by $h$. As the hermitian form $\overline{h}$ is non-degenerate, the intersection of all such subspaces is trivial, implying
\begin{equation} \label{schnitt} \Lambda=\bigcap_{\{\tilde\Lambda\in\mathcal{L}^{\max}\ |\ \Lambda\subset\tilde\Lambda\}} \tilde\Lambda. \end{equation}
If we assume that all irreducible components $\mathcal{V}(\tilde\Lambda)$ of $\mathcal{N}(\mathbb{F})$ are in $\mathcal{Z}(x)(\mathbb{F})$, then all $\tilde\Lambda\supset\Lambda$ of type $t_{\max}$ satisfy $\tilde\Lambda_\parallel=L_{(1,0)}$ by Lemma \ref{einslambda}. Thus $\Lambda_\parallel=L_{(1,0)}$ and Remark \ref{thirtysix} implies the ``if'' part.\\
For the ``only if'' part, note that
\[ \Lambda_\bullet=\Lambda_\parallel\oplus\Lambda_\perp\Rightarrow\Lambda_\parallel=L_{(1,0)}\Rightarrow p^{-1}x\in\Lambda\Leftrightarrow x\in p\Lambda \]
and that, for any vertex $\tilde\Lambda$ containing $\Lambda$, we have the chain of inclusions
\[ p\Lambda\subseteq p\tilde\Lambda\subseteq\tilde\Lambda^\sharp\subseteq\Lambda^\sharp. \]
Thus $x\in p\Lambda$ implies $x\in\tilde\Lambda^\sharp$, which means $\tilde\Lambda\in \mathcal{S}(x)$. The statement on irreducible components is simply obtained by specializing to $t(\tilde\Lambda)=t_{\max}$ in these considerations.
\end{proof}

\begin{prop} \label{kodim1}
Let $\Lambda\in \mathcal{S}(x)$ be of type $t_{\max}-2$. Then if $\Lambda_\bullet=\Lambda_\parallel\oplus\Lambda_\perp$, i.e.\ if $\Lambda$ comes from a vertex of $\mathfrak{B}(SU(C_\perp),\mathbb{Q}_p)$, then all the irreducible components of $\mathcal{N}_\textnormal{red}$ passing through the codimension 1 stratum $\mathcal{N}_\Lambda$ are in $\mathcal{Z}(x)_\textnormal{red}$.

Otherwise, there is exactly one irreducible component $\mathcal{N}_{\tilde\Lambda}$ of $\mathcal{N}_\textnormal{red}$ which contains $\mathcal{N}_{\Lambda}$ and is contained in $\mathcal{Z}(x)_\textnormal{red}$.
\end{prop}

\begin{proof}
The first assertion is the case $t(\Lambda)=t_{\max}-2$ of the last proposition.

For the second one, first we observe that there is (in any case) an irreducible component $\mathcal{N}_{\tilde\Lambda}$ contained in $\mathcal{Z}(x)_{\textnormal{red}}$ and containing $\mathcal{N}_\Lambda$. This is because $\mathcal{Z}(x)_\textnormal{red}$ has pure codimension 0 in $\mathcal{N}_\textnormal{red}$.
Now assume that two irreducible components of $\mathcal{Z}(x)_\textnormal{red}$ pass through $\mathcal{N}_\Lambda$, i.e.
\[ \mathcal{V}(\Lambda)\subseteq\mathcal{V}(\Lambda^{(1)})\cap\mathcal{V}(\Lambda^{(2)}) \]
for some distinct $\Lambda^{(1)}$, $\Lambda^{(2)}\in \mathcal{S}(x)^{\max}$. Then the inclusion is an equality, because $\mathcal{N}_\Lambda$ has codimension 1 in $\mathcal{N}_\textnormal{red}$ and the intersection $\mathcal{N}_{\Lambda^{(1)}}\cap\mathcal{N}_{\Lambda^{(2)}}$ is irreducible. Therefore by \cite{vw}, Thm.\ 4.1(2), we have $\Lambda=\Lambda^{(1)}\cap\Lambda^{(2)}$ and in particular
\[ \Lambda_\parallel=\Lambda^{(1)}_\parallel\cap\Lambda^{(2)}_\parallel. \]
Both irreducible components $\mathcal{N}_{\Lambda^{(1)}}$, $\mathcal{N}_{\Lambda^{(2)}}$ are in $\mathcal{Z}(x)_\textnormal{red}$, thus $\Lambda^{(i)}_\parallel=L_{(1,0)}$ for $i=1,2$ and therefore $\Lambda_\parallel=L_{(1,0)}$. By Remark \ref{thirtysix}, we have $\Lambda_\bullet=\Lambda_\parallel\oplus\Lambda_\perp$.
\end{proof}

\begin{remark}
Note that, in the situation of the proposition, the set of irreducible components of $\mathcal{N}_\textnormal{red}$ containing $\mathcal{N}_\Lambda$ is in bijection with the set of isotropic lines of the $(n-t(\Lambda))$-dimensional non-degenerate hermitian $\mathbb{F}_{p^2}$-vector space $(V':=\Lambda^\sharp/p\Lambda,\overline{h})$.

Vollaard and Wedhorn computed the cardinality of this set in \cite{vw}, Example 4.6. The answer is $p+1$ if $n$ is odd, $p^3+1$ if $n$ is even.
\end{remark}

\begin{expl}[$n=3$] \label{drei}
We now examine the first non-trivial case $n=3$. Vollaard computed the simplicial complex of $\mathfrak{B}(SU(C_\perp,ph),\mathbb{Q}_p)$ in \cite{voll}, Prop.\ 3.8. We cite the following results:
\begin{itemize}
\item $\mathcal{N}_\textnormal{red}$ is pure of dimension 1.
\item The irreducible components of $\mathcal{N}_\textnormal{red}$ are isomorphic to plane Fermat curves $(x_0^{p+1}+x_1^{p+1}+x_2^{p+1}=0)$ in $\mathbb{P}_\mathbb{F}^2$.
\item Irreducible components intersect precisely in the $\mathbb{F}_{p^2}$-valued points, of which there are $p^3+1$ on a given irreducible component.
\item In such a point, precisely $p+1$ irreducible components intersect.
\end{itemize}
For special homomorphisms $x$ of valuation 0, Terstiege shows (\cite{ter}, Prop.\ 2.1(1)) that $\mathcal{Z}(x)_\textnormal{red}$ consists of a single $\mathbb{F}$-valued point which is $\tau$-invariant. This agrees with the results of Section 6. Indeed, if $r(x)=0$, then $C_\perp$ is two-dimensional, with $\textnormal{ord}\,\textnormal{det}(C_\perp,h)$ odd. Therefore, only one type (namely type 1) of vertices occurs in $\mathfrak{B}(SU(C_\perp,ph),\mathbb{Q}_p)$ and the latter is connected, hence consists of a single point. By Thm.\ \ref{prime}(1), $\mathcal{S}(x)$ consists of a single point, which by Prop.\ \ref{bts} implies Terstiege's statement.

Let now $x$ be a special homomorphism of valuation 1. Terstiege showed in \cite{ter}, Prop.\ 2.1(3), that $\mathcal{Z}(x)_\textnormal{red}$ is connected of pure dimension 1, and that, given $\Lambda\in \mathcal{S}(x)^{\max}$, there are precisely $p+1$ out of the $p^3+1$ type 1 vertices $\tilde\Lambda\subset\Lambda$ for which all $p+1$ irreducible components passing through $\mathcal{N}_{\tilde\Lambda}$ belong to $\mathcal{Z}(x)_\textnormal{red}$. He also showed that for the other $p^3-p$ vertices $\tilde\Lambda\subset\Lambda$, the stratum $\mathcal{N}_\Lambda$ is the only irreducible component out of those passing through $\mathcal{N}_{\tilde\Lambda}$ which belongs to $\mathcal{Z}(x)_\textnormal{red}$. We will deduce this from our general results as an illustration.

In fact, connectedness has just been proven in general for special homomorphisms of valuation 1. Fixing a vertex $\Lambda$ of maximal type 3 in $\mathcal{S}(x)^{\max}$, we also saw that $\Lambda_\bullet$ decomposes as $\Lambda_\bullet=\Lambda_\parallel\oplus\Lambda_\perp$, with $\Lambda_\perp$ a vertex of type 2 in $\mathcal{L}_\perp$. The vertices $\tilde\Lambda$ of type 1 contained in $\Lambda$ fall into two categories, those for which $\tilde\Lambda_\parallel=L_{(0,0)}$ and those for which $\tilde\Lambda_\parallel=L_{(1,0)}$. By Remark \ref{thirtysix} and Prop.\ \ref{kodim1}, the latter are exactly those for which all irreducible components passing through $\mathcal{N}_{\tilde\Lambda}$ are in $\mathcal{Z}(x)_\textnormal{red}$, while in the former case, the closed stratum $\mathcal{N}_\Lambda$ is the only irreducible component of $\mathcal{Z}(x)_\textnormal{red}$ passing through $\mathcal{N}_{\tilde\Lambda}$.

We thus have to count the $\tilde\Lambda$ as above which satisfy $\tilde\Lambda_\parallel=L_{(1,0)}$. By Remark \ref{thirtysix}, those are precisely those for which $\tilde\Lambda_\bullet$ decomposes as $\tilde\Lambda_\bullet=\tilde\Lambda_\parallel\oplus\tilde\Lambda_\perp$, with $\tilde\Lambda_\perp$ a vertex of type 0 in $\mathcal{L}_\perp$, contained in the vertex $\Lambda_\perp$ of type 2. Now the argument of \ref{muff} does not in any way depend on the order of determinant. Thus we may just count isotropic subspaces in the 2-dimensional hermitian $\mathbb{F}_{p^2}$-vector space $\Gamma/\Gamma^\sharp=\Gamma/p\Gamma$, where $\Gamma:=(\Lambda_\perp)_0^{\tau=1}$. It follows from \cite{vw}, Ex.\ 4.6, that there are precisely $p+1$ of those, yielding Terstiege's result.

This argument also shows that the ``superspecial'' stratum $\mathcal{Z}(x)^{(1,0)}$ consists exactly of the points where the irreducible components of $\mathcal{Z}(x)_\textnormal{red}$ intersect.
\end{expl}

\begin{expl}[$n=4$]
The other one-dimensional case is $n=4$. For this case, the global structure of the reduced locus $\mathcal{N}_\textnormal{red}$ was computed by Vollaard and Wedhorn (\cite{vw}, Ex.\ 4.8), with the following results:
\begin{itemize}
\item $\mathcal{N}_\textnormal{red}$ is pure of dimension 1.
\item The irreducible components $\mathcal{N}_\textnormal{red}$ are isomorphic to plane Fermat curves $(x_0^{p+1}+x_1^{p+1}+x_2^{p+1}=0)$ in $\mathbb{P}_\mathbb{F}^2$.
\item Irreducible components intersect precisely in the $\mathbb{F}_{p^2}$-valued points, of which there are $p^3+1$ on a given irreducible component.
\item In such a point, precisely $p^3+1$ irreducible components intersect.
\end{itemize}
In other words, the irreducible components are isomorphic to those in the case $n=3$, but there are more of them passing through a given zero-dimensional Bruhat-Tits stratum.

We know from the results of the last section that, given a special homomorphism $x\in\mathbb{V}$ of valuation 0, we have an isomorphism of formal schemes $\mathcal{Z}(x)\cong\mathcal{N}(1,2)$. By \ref{nullbtg}, this bijection is compatible with the Bruhat-Tits stratifications of $\mathcal{Z}(x)_\textnormal{red}$ and $\mathcal{N}(1,2)_\textnormal{red}$. Thus computing the combinatorics of $\mathfrak{B}(SU(C_\perp,ph),\mathbb{Q}_p)$ shows:

\begin{prop}
Let $x\in\mathbb{V}$ be of valuation 0. Then $\mathcal{Z}(x)_\textnormal{red}$ is of pure codimension 0. Given any vertex $\Lambda$ of type 1 in $\mathcal{S}(x)$, precisely $p+1$ out of the $p^3+1$ irreducible components of $\mathcal{N}_\textnormal{red}$ passing through $\mathcal{N}_\Lambda$ belong to $\mathcal{Z}(x)_\textnormal{red}$.
\end{prop}

Now assume $r(x)=v(h(x,x))=1$. We have proved that $\mathcal{Z}(x)_\textnormal{red}$ is connected. Let $\Lambda\in\mathcal{L}$ be of maximal type 3. By the same argument as in the case $n=3$, one shows that out of the $p^3+1$ vertices $\tilde\Lambda$ of type 1 contained in $\Lambda$, precisely $p+1$ have the property that $\tilde\Lambda_\bullet$ decomposes as $\tilde\Lambda_\bullet=\tilde\Lambda_\parallel\oplus\tilde\Lambda_\perp$. Again, one uses Prop.\ \ref{kodim1} and computes the ``superspecial'' stratum as for $n=3$ to get:

\begin{prop}
Let $x\in\mathbb{V}$ be of valuation 1. Then $\mathcal{Z}(x)(\mathbb{F})$ is of pure codimension 0. All lattices in the ``superspecial'' stratum $\mathcal{Z}(x)^{(1,0)}$ come from vertices in $\mathcal{L}$, i.e.\ are $\tau$-invariant. Out of the $p^3+1$ type 1 vertices contained in any given type 3 vertex $\Lambda\in \mathcal{S}(x)^{\max}$, precisely $p+1$ give rise to points in $\mathcal{Z}(x)^{(1,0)}$ via \ref{lambda}.

Let $\Lambda\in\mathcal{S}(x)^{\max}$. Let $\tilde\Lambda$ be a vertex of type 1, contained in the type 3 vertex $\Lambda$. If $\tilde\Lambda_\bullet$ belongs to $\mathcal{Z}(x)^{(1,0)}$, then all $p^3+1$ irreducible components of $\mathcal{N}_\textnormal{red}$ passing through $\mathcal{N}_{\tilde\Lambda}$ belong to $\mathcal{Z}(x)_\textnormal{red}$. If $\tilde\Lambda_\bullet$ belongs to $\mathcal{Z}(x)^{(0,0)}$, then $\mathcal{N}_\Lambda$ is the only irreducible component of $\mathcal{Z}(x)_\textnormal{red}$ passing through $\mathcal{N}_{\tilde\Lambda}$.
\end{prop}

\end{expl}

\section{Recursion step: The case $r>1$}

In this section, we prove Thm.\ \ref{prime}(3), i.e.\ compute $\mathcal{S}(x)$ in terms of $\mathcal{S}(p^{-1}x)$ under the assumption $r=(x)=v(h(x,x))>1$. As a consequence, we obtain a formula for $\mathcal{S}(x)$ for any $r\geq 0$ and a complete proof of Thm.\ \ref{holygrail} for $m=1$ by induction on $\lfloor r/2\rfloor$, starting with $\lfloor r/2\rfloor=0$, i.e.\ with the two cases treated in the previous two sections. In this section, we assume $r>1$ unless specified otherwise.

\begin{fliess}
We may reformulate Lemma \ref{ausbr} as the inclusion of subsets of $\mathcal{L}$,
\[ \{\Lambda\in\mathcal{L}^{\max}\ |\ d(\Lambda,\mathcal{S}(p^{-1}x))\leq 1\}\subseteq \mathcal{S}(x). \]
We will prove Lemma \ref{upperbound} which states that this inclusion is an equality, thus giving an explicit description of the set $\mathcal{S}(x)$ in terms of $\mathcal{S}(p^{-1}x)$ and proving Thm.\ \ref{prime}(3). (Note that we know from Prop.\ \ref{dimension} that $\mathcal{Z}(x)_\textnormal{red}$ is of pure codimension 0. Thus, all irreducible components of $\mathcal{Z}(x)_\textnormal{red}$ correspond to vertices of maximal type, i.e.\ $\mathcal{S}(x)^{\max}$ determines $\mathcal{S}(x)$.)
\end{fliess}

\begin{lem}\label{tomate}
Let $\Lambda\in \mathcal{S}(x)^\textnormal{max}$. Then either $\Lambda\in\mathcal{S}(p^{-1}x)$ or $\Lambda_\parallel=L_{(1,0)}$.
\end{lem}

\begin{proof}
Let $\Lambda\in \mathcal{S}(x)^{\max}$. From the general considerations in \ref{lambda}, writing $\Lambda_\parallel=L_{(a,b)}$, we get the three possibilites $b\geq 1$ (i.e.\ $\Lambda\in\mathcal{S}(p^{-1}x)$ by Lemma \ref{leo}), $(a,b)=(1,0)$ and $(a,b)=(0,0)$. We have to show that the last case does not occur.

If $n$ is odd, one proceeds exactly as in the proof of Lemma \ref{einslambda}.
If $n$ is even, we run into the same difficulties as in that proof. We assume $\Lambda_\parallel=L_{(0,0)}$ and proceed as in the mentioned proof to get the inclusion diagram
\[
\begin{matrix}
 {}&{}&\Lambda_{\parallel,1}^\vee\oplus\Lambda_{\perp,1}^\vee\\
 {}&{}&\cup\\
 F\Lambda_1&\overset{n-1}{\subset}&\Lambda_0\\
 \cup&{}&\cup\\
 F(\Lambda_{\parallel,1}\oplus\Lambda_{\perp,1})&\overset{\alpha}{\subset}&\Lambda_{\parallel,0}\oplus\Lambda_{\perp,0},
 \end{matrix}
 \qquad\qquad
 \begin{matrix}
 \Lambda_{\parallel,0}^\vee\oplus\Lambda_{\perp,0}^\vee{}&{}\\
 \cup&{}&{}\\
 \Lambda_1&\overset{1}{\supset}&F\Lambda_0\\
 \cup&{}&\cup\\
 \Lambda_{\parallel,1}\oplus\Lambda_{\perp,1}&\overset{\beta}{\supset}&F(\Lambda_{\parallel,0}\oplus\Lambda_{\perp,0}),
 \end{matrix}
\]
where the indices for all inclusions in the columns are $r$. It follows that $\alpha=n-1$ and $\beta=1$, which implies the signature condition
\begin{equation} \label{sigzwo} p\Lambda_{\perp,0}=F\Lambda_{\perp,1}\overset{\text{\tiny{\itshape n-\normalfont 1}}}{\subset}\Lambda_{\perp,0},\quad p\Lambda_{\perp,1}\overset{\text{\tiny{\itshape n-\normalfont 1}}}{\subset}F\Lambda_{\perp,0}=\Lambda_{\perp,1}. \end{equation}
On the other hand, $\Lambda_\parallel^\vee/\Lambda_\parallel$ is the $W$-module
\[ L_{(r,r)}/L_{(0,0)}=(W/p^rW)\overline{p^{-r}x_0}\oplus(W/p^rW)\overline{p^{-r}x_1}, \]
endowed with a $\sigma$-linear endomorphism $\overline{F}$ induced by $F$. As usual, we write $\overline{F}_0$, resp.\ $\overline{F}_1$, for the restriction of $\overline{F}$ to $(\Lambda_\parallel^\vee/\Lambda_\parallel)_0=\Lambda_{\parallel,1}^\vee/\Lambda_{\parallel,0}$, resp.\ $(\Lambda_\parallel^\vee/\Lambda_\parallel)_1$. Since $FL_{(r,r)}=L_{(r,r-1)}$ and $FL_{(1,0)}=L_{(0,0)}$, we get that, while $\overline{F}_1$ is an isomorphism, $\overline{F}_0$ has the non-trivial kernel
\[ \ker(\overline{F_0})=(p^{r-1}W/p^rW)\overline{p^{-r}x_0}. \]
We have the same notions of $\overline{F}_0$ and $\overline{F}_1$ on $\Lambda_\perp^\vee/\Lambda_\perp$. By Lemma \ref{fortytwo}, we have a $W$-linear isomorphism
\[ \gamma:\Lambda_\parallel^\vee/\Lambda_\parallel\rightarrow\Lambda_\perp^\vee/\Lambda_\perp, \]
and $\gamma$ is $\overline{F}$-equivariant.

Thus, on $\Lambda_\perp^\vee/\Lambda_\perp$, we have that $\overline{F}_1$ is an isomorphism, while the kernel of $\overline{F}_0$ is a free $(p^{r-1}W/p^rW)$-module of rank 1. The preimage of this module in $\Lambda_{\perp,1}^\vee$ is a $W$-lattice $\Gamma$ in $N_{\perp,0}$ satisfying $\Lambda_{\perp,0}\overset{\text{\tiny 1}}{\subset}\Gamma$. By construction of $\Gamma$, we have
\[ F\Gamma\subseteq\Lambda_{\perp,1}. \]
But $F$ is injective, which implies that $F\Lambda_{\perp,0}$ is a proper sublattice of $F\Gamma$. This contradicts the signature condition \eqref{sigzwo}.
\end{proof}

\begin{remark}
We will now translate this result to the $(C,h)$ language. Consider a vertex $\Lambda$ of maximal type $t_{\max}$ in $\mathcal{L}$. By definition and our ``translation lemma'' \ref{leo}, the assumption of Lemma \ref{tomate} (i.e.\ $\Lambda_\parallel=L_{(a,b)}$ with $b\geq 0$) translates into the statement that $x\in\Lambda^\sharp$. The assertion of the lemma, namely that in this case $a\geq 1$, translates into 
\[ x\in\Lambda^\sharp\Rightarrow p^{-1}x\in\Lambda. \]
Of course, this is trivial for odd $n$, but not at all obvious for even $n$.
\end{remark}

\begin{lem} \label{upperbound}
Let $\Lambda\in \mathcal{S}(x)^{\max}$. Then either $\Lambda\in\mathcal{S}(p^{-1}x)$ or
\[ \mathcal{N}_\Lambda\cap\mathcal{Z}(p^{-1}x)_\textnormal{red}=\mathcal{N}_{\tilde{\Lambda}} \]
for some $\tilde\Lambda\in\mathcal{L}$ of type $t_{\max}-2$.
\end{lem}

\begin{proof}
Let $\Lambda\in \mathcal{S}(x)^{\max}$, i.e.\ $x\in\Lambda^\sharp$. The previous lemma states that $p^{-1}x\in\Lambda$.

Let $(V=\Lambda/\Lambda^\sharp,\overline{ph})$ be the hermitian $\mathbb{F}_{p^2}$-vector space defined in \ref{muff}. There, we stated that vertices $\tilde\Lambda$ of type $t$ contained in $\Lambda$ (i.e.\ $\mathcal{N}_{\tilde\Lambda}\subseteq\mathcal{N}_\Lambda$) correspond to $((t_{\max}-t)/2)$-dimensional isotropic subspaces $U=\tilde\Lambda^\sharp/\Lambda^\sharp$ of $V$. Furthermore, $p^{-1}x\in\tilde\Lambda^\sharp$ is equivalent to $\overline{p^{-1}x}\in U$, where $\overline{p^{-1}x}$ is the image of $p^{-1}x$ in $V$.

But $\overline{p^{-1}x}$ is isotropic (since $v(h(x,x))\geq 2$), hence is either zero (which means $p^{-1}x\in\Lambda^\sharp$, i.e.\ $\Lambda\in\mathcal{S}(p^{-1}x)$) or spans an isotropic line $U$. In the latter case, the intersection $\mathcal{N}_\Lambda\cap\mathcal{Z}(x)_\textnormal{red}$ is of the form $\mathcal{N}_{\tilde{\Lambda}}$, where $\tilde{\Lambda}$ is the dual of the preimage of $U$ in $\Lambda$ and thus a vertex of type $t_{\max}-2$.
\end{proof}

\begin{fliess} \label{rekursion}
Combining this with Lemma \ref{ausbr}, we get an explicit formula describing the irreducible components of $\mathcal{Z}(x)_\textnormal{red}$ in terms of those of $\mathcal{Z}(p^{-1}x)_\textnormal{red}$, namely
\[ \mathcal{S}(x)^{\max}=\{\Lambda\in\mathcal{L}^{\max}\ |\ d(\Lambda,\mathcal{S}(p^{-1}x))\leq 1\}. \]
This proves Thm.\ \ref{prime}(3).
\end{fliess}

\begin{cor}
Let $x\in\mathbb{V}$ be of valuation $r>1$. Then
\[ \mathcal{S}(x)^{\max}=\{\Lambda\in\mathcal{L}^{\max}\ |\ d(\Lambda,\mathcal{S}(p^{-\lfloor r/2\rfloor}x))\leq\lfloor r/2\rfloor\}, \]
with $\mathcal{S}(p^{-\lfloor r/2\rfloor}x)$ being known from Thm.\ \ref{prime}(1) and (2).
\end{cor}

\begin{proof}
Iterate the proposition.
\end{proof}

\begin{cor}
Assumptions as in the last corollary. Then Thm.\ \ref{holygrail} holds true for $x$, i.e.\ $\mathcal{Z}(x)_\textnormal{red}$ is connected.
\end{cor}

\begin{proof}
It is enough to prove connectedness of $\mathcal{S}(x)$ as a simplicial subcomplex of $\mathcal{L}$. We will show that connectedness of $\mathcal{S}(p^{-1}x)$ implies connectedness of $\mathcal{S}(x)$. Since Thm.\ \ref{holygrail} is known to hold for special homomorphisms of valuation 0 or 1, this implies our claim by induction on $\lfloor r/2\rfloor$. Thus assume that $\mathcal{S}(p^{-1}x)$ is connected.

Let $\Lambda$, $\tilde\Lambda$ be in $\mathcal{S}(x)$. We have to construct a sequence
\[ \Lambda=\Lambda^{(0)},\Lambda^{(1)},,\dots,\Lambda^{(s)}=\tilde\Lambda \]
of vertices in $\mathcal{S}(x)$ with the property that for any $1\leq i\leq s$, the vertices $\Lambda^{(i-1)}$ and $\Lambda^{(i)}$ neighbour each other.

Since $\mathcal{Z}(x)_\textnormal{red}$ is pure of codimension 0 in $\mathcal{N}_\textnormal{red}$, one finds $\Lambda^{(1)}$, $\tilde\Lambda^{(1)}\in \mathcal{S}(x)^{\max}$ containing $\Lambda$, resp.\ $\tilde\Lambda$. By Lemma \ref{upperbound}, one finds $\Lambda^{(2)}$, $\tilde\Lambda^{(2)}\in\mathcal{S}(p^{-1}x)$ contained in $\Lambda^{(1)}$, resp.\ $\tilde\Lambda^{(1)}$. By assumption, $\mathcal{S}(p^{-1}x)$ is connected, hence one finds a path
\[ \Lambda^{(2)},\Lambda^{(3)},\dots,\Lambda^{(s-2)}=\tilde\Lambda^{(2)} \]
in $\mathcal{S}(p^{-1}x)$ connecting $\Lambda^{(2)}$ and $\tilde\Lambda^{(2)}$. Then set $\Lambda^{(s-1)}:=\tilde\Lambda^{(1)}$ and $\Lambda^{(s)}:=\tilde\Lambda$.
\end{proof}

\section{Intersections of special cycles}

Finally we consider the case $m>1$. Let $\underline{x}=(x_1,\dots,x_m)\in\mathbb{V}$ be according to our general assumptions and conventions of \ref{fleshwound}. The aim of this section is to show Thm.\ \ref{vielexprime} for
\[ \mathcal{S}(\underline{x})=\mathcal{S}(x_1)\cap\dots\cap\mathcal{S}(x_m) \]
and then to show Thm.\ \ref{holygrail} in full generality.

\begin{fliess}
Recall that, according to \ref{fleshwound}, we assume the $x_i$ to be perpendicular to each other w.r.t.\ $h$ and to be of nonnegative finite valuation $r_i:=r(x_i)=v(h(x_i,x_i))$. We will furthermore simplify notation by assuming that the $x_i$ are ordered increasingly by valuation. For any nonnegative integer $r\geq 0$, we fix the following notations:
\begin{align*}
m_r &:=\max\{i\ |\ r_i<r\}, \\
C_r &:=(\textnormal{span}_{\mathbb{Q}_{p^2}}(x_1,\dots,x_{m_r}))^\perp, \\
\mathcal{L}_r^{(t)} &:=\{\Lambda\subset C_r\ \textnormal{a}\ \mathbb{Z}_{p^2}\textnormal{-lattice}\ |\ p\Lambda\subseteq\Lambda^\sharp\overset{\text{\tiny{\itshape t}}}{\subseteq}\Lambda \}, \\
\mathcal{L}_r &:=\bigcup\nolimits_t \mathcal{L}_r^{(t)}.
\end{align*}
(Note that all those notions depend not only on $r$, but also on $\underline{x}$.)

$C_r$ is an $(n-m_r)$-dimensional $\mathbb{Q}_{p^2}$-vector space endowed with a non-degenerate hermitian form induced by $h$. Since $h$ has odd order of determinant on $C$, the order of determinant of $h$ on $C$ is odd if and only if the number of indices $i\leq m_r$ for  which $r_i\equiv 1\ (2)$ is even. Note that we have
\[ C_r=\textnormal{span}_{\mathbb{Q}_{p^2}}(x_{m_r+1},\dots,x_{m_{r+1}})\oplus C_{r+1}. \]
Also, $C_0=C$, whereas in the case $m=1$ and $r>r(x)$, the space $C_r$ is simply $C_\perp$ as defined in Section 5.

By Thm.\ 3.6 of \cite{voll}, for any $r$, the set $\mathcal{L}_r$ is in bijection with the set of vertices of $\mathfrak{B}(SU(C_r,ph),\mathbb{Q}_p)$. We endow $\mathcal{L}_r$ with the simplicial complex structure induced by this bijection, i.e.\ two distinct vertices $\Lambda$, $\tilde\Lambda$ neighbour each other in the simplicial complex $\mathcal{L}_r$ if and only if one of them contains the other. By Bruhat-Tits theory, the simplicial complex $\mathcal{L}_r$ is connected. Furthermore, we have distance functions on $\mathcal{L}_r$ defined by analogy with \ref{keepyourdistance}, which we denote by $d_r$. Of course, $\mathcal{L}_0=\mathcal{L}$, while for $m=1$ and $r>r(x)$, the simplicial complex $\mathcal{L}_r$ is isomorphic to the complex $\mathcal{L}_\perp$ which we used in Sections 6 and 7 via the bijection $\Lambda\mapsto\Lambda_\bullet$ introduced in \ref{lambda}.
\end{fliess}

\begin{fliess}
Let $s$ be any nonnegative integer. Let
\[ N_{(s)}:=(\textnormal{span}_{W_\mathbb{Q}}(x_1,\dots,x_{m_{2s}},F^{-1}x_1,\dots,F^{-1}x_{m_{2s}}))^\perp, \]
where the orthogonal complement is taken w.r.t.\ the symplectic form $\langle.,.\rangle$. Given $\Lambda\in\mathcal{L}_{2s}$, we construct a self-dual $W$-lattice $\Lambda_\bullet$ in $N_{(s)}$ following the procedure of \ref{lambda}. We furthermore write for $1\leq i\leq m$:
\[ L_i:=p^{-a}Wx_i\oplus p^{-b}WF^{-1}x_i, \]
where $(a,b)=(r_i/2,r_i/2)$ if $r_i$ is even, $(a,b)=\big((r_i+1)/2,(r_i-1)/2\big)$ if $r_i$ is odd. This is the choice of $a$, $b$ for which $L_i$ is self-dual w.r.t.\ $\langle.,.\rangle$. Now, for $\Lambda\in\mathcal{L}_{2s+2}$, set
\[ \Phi_s(\Lambda):=\left(\left(\bigoplus_{i=m_{2s}+1}^{m_{2s+2}}L_i\right)\oplus\Lambda_\bullet\right)_0^{\tau=1}\subset C_{2s}. \]
Checking that this defines a vertex in $\mathcal{L}_{2s}$ is straightforward, using the self-duality of $L_i$ and $\Lambda_\bullet$ w.r.t.\ $\langle .,.\rangle$. Thus we have an injective morphism of simplicial complexes
\[ \Phi_s:\mathcal{L}_{2s+2}\rightarrow\mathcal{L}_{2s}. \]
One observes that, if $\Lambda\in\mathcal{L}_{2s+2}$ is of type $t$, the type of $\Phi_s(\Lambda)$ is $t+m_{2s+2}-m_{2s+1}$.
\end{fliess}

\begin{defin} \label{allesdoof}
Let $0\leq s\leq\lfloor r_m/2\rfloor$. Set
\[ S_s=S_s(\underline{x}):=\{\Lambda\in\mathcal{L}_{2s}\ |\ p^{-s}x_i\in\Lambda^\sharp\ \forall\,i>m_{2s}\}. \]
For $1\leq s\leq\lfloor r_m/2\rfloor+1$, we set
\[ S_s'=S_s'(\underline{x}):=\{\Lambda\in\mathcal{L}_{2s}\ |\ p^{-s+1}x_i\in\Lambda^\sharp\ \forall\,i>m_{2s}\}. \]
\end{defin}

In particular, $S_0$ is the simplicial complex $\mathcal{S}(\underline{x})$ which we are originally interested in.

On the other hand, for the maximal occurring index $s=\lfloor r_m/2\rfloor+1$, the condition defining $S_s'$ is empty, i.e.\ $S_s'$ simply equals $\mathcal{L}_{2s}=\mathcal{L}_{r_m+1}$ as simplicial complexes. Thm.\ 3.6 of \cite{voll} provides us with an explicit description of $\mathcal{L}_{r_m+1}$ as the simplicial complex of a specific Bruhat-Tits building.

This case will serve as the starting point for explicitly computing $S_s$ and $S_s'$ (and thus $\mathcal{S}(\underline{x})$) by downwards induction on $s$. We now give the algorithm:

\begin{enumerate}
\addtocounter{enumi}{-1}
\item Set $s:=\lfloor r_m/2\rfloor$. Set $V_{s+1}':=\mathcal{L}_{2s+2}=\mathcal{L}_{r_m+1}$.\smallskip
\item Set $V_s:=\{\Lambda\in\mathcal{L}_{2s}\ |\ \exists\ \tilde\Lambda\in \Phi_s(V_{s+1}'):\ \Lambda\subseteq\tilde\Lambda\}$.\smallskip
\item Set $V_s':=\{\Lambda\in\mathcal{L}_{2s}\ |\ \exists\ \tilde\Lambda\in\mathcal{L}_{2s}^{\max}:\ \Lambda\subseteq\tilde\Lambda,\,d_{2s}(\tilde\Lambda,V_s)\leq 1\}$.\smallskip
\item If $s>0$, continue with Step 1, replacing $s$ by $s-1$. If $s=0$, stop.
\end{enumerate}

\begin{thm}\label{vielex}
The equalities
\[ S_s=V_s,\ S_s'=V_s' \]
hold for any $0\leq s\leq\lfloor r_m/2\rfloor$ resp.\ $1\leq s\leq\lfloor r_m/2\rfloor+1$.
\end{thm}

We will prove the theorem (which immediately implies Thm.\ \ref{vielexprime}) by downwards induction on $s$, using the obvious equality $S_{\lfloor r_m/2\rfloor+1}=V_{\lfloor r_m/2\rfloor+1}$ as starting point and doing the induction step by proving the implications
\[ S_{s+1}'=V_{s+1}'\Rightarrow S_s=V_s\Rightarrow S_s'=V_s'. \]
We will now sketch the proof of these two implications. First, we show Lemma \ref{flausch}, which states that $S_{s+1}'=V_{s+1}'$ implies the ``easy'' inclusion $V_s\subseteq S_s$. The following lemma \ref{fluff} states that the ``hard'' inclusion holds for $s=0$, i.e.\ that $S_1'=V_1'$ implies $S_0=V_0$. This special case has to be done first because we have no ``pure-dimensionality'' statements for any $S_s$ or $S_s'$, except for $S_0$ (Prop.\ 3.12). The proof goes by applying the results of Sections 6 and 7. The general case is then reduced to this special case in the proof of Lemma \ref{knuddel}. This concludes the proof of the first implication.

For the second implication, we first prove Lemma \ref{schmieg}, which states that $S_s=V_s$ implies $V_s'\subseteq S_s'$. This is a straightforward generalization of Lemma \ref{ausbr}. We then prove Lemma \ref{kuschel}, which states that $S_s'$ is ``pure of codimension 0'' in $\mathcal{L}_{2s}$ (this will be made precise in the statement of the lemma), by reduction to Lemma \ref{fluff}. Having proven this, we obtain the ``hard'' inclusion $S_s'\subseteq V_s'$ as a straightforward generalization of the results of Section 8. This is done in Lemma \ref{schnuff} and concludes the proof of Thm.\ \ref{vielex}.

\begin{lem} \label{flausch}
Let $0\leq s\leq\lfloor r_m/2\rfloor$ be arbitrary. Assume that $S_{s+1}'=V_{s+1}'$. Then $V_s\subseteq S_s$.
\end{lem}

\begin{proof}
By definition of $V_s$, it is enough to prove $\Phi_s(V_{s+1}')\subseteq S_s$. Let $\Lambda\in\Phi_s(V_{s+1}')$. This means
\[ \Lambda_\bullet=\left(\bigoplus_{i=m_{2s}+1}^{m_{2s+2}}L_i\right)\oplus\tilde\Lambda_\bullet \]
for some $\tilde\Lambda\in V_{s+1}'$. By definition of $m_{2s}$ and $m_{2s+2}$, one has $(L_i)_1=p^{-s}WF^{-1}x_i$ for $m_{2s}<i\leq m_{2s+2}$. It follows that
\[ \Lambda^\sharp=(F\Lambda_1)^{\tau=1}=\left(\bigoplus_{i=m_{2s}+1}^{m_{2s+2}}p^{-s}\mathbb{Z}_{p^2}x_i\right)\oplus\tilde\Lambda^\sharp \]
contains $p^{-s}x_i$ for $m_{2s}<i\leq m_{2s+2}$. But the $p^{-s}x_i$ with $i>m_{s+1}$ are contained in $\tilde\Lambda^\sharp$ since $\tilde\Lambda$ is in $S_{s+1}'=V_{s+1}'$ by assumption. Thus these $p^{-s}x_i$ are also contained in $\Lambda^\sharp$.
\end{proof}

\begin{lem} \label{fluff}
Assume that $S_1'=V_1'$. Then $S_0\subseteq V_0$.
\end{lem}

\begin{proof}
Recall from the definitions above that
\[ S_0=\{\Lambda\in\mathcal{L}\ |\ x_i\in\Lambda^\sharp\ \forall\,1\leq i\leq m\}=\mathcal{S}(\underline{x}). \]
By Prop.\ \ref{dimension}, $\mathcal{Z}(\underline{x})_\textnormal{red}$ is pure of dimension $\lfloor (n-m_1-1)/2\rfloor$, i.e.\ every vertex in $\mathcal{S}(\underline{x})$ is contained in some vertex in $\mathcal{S}(\underline{x})^{(t_0)}:=T(\underline{x})\cap\mathcal{L}^{t_0}$, where $t_0$ is the largest odd integer $\leq n-m_1$. One checks that $t_0-m_2+m_1$ is the maximal type of vertices in $\mathcal{L}_2$.

Let $\Lambda\in\mathcal{S}(\underline{x})$. Using the results on a single special homomorphism of valuation 0, in particular \ref{nullbtg}, we get for any $t\leq m_1$:
\[ (\Lambda_\bullet)\cap(W_\mathbb{Q}x_i\oplus W_\mathbb{Q}F^{-1}x_i)=Wx_i\oplus WF^{-1}x_i=L_i, \]
which is self-dual w.r.t.\ $\langle.,.\rangle$. By Remark \ref{thirtysix}, this implies that $\Lambda_\bullet$ decomposes orthogonally as $\Lambda_\bullet=\left(\bigoplus_{i=1}^{m_1}L_i\right)\oplus\Lambda_\bullet^{(1)}$ for some $\Lambda^{(1)}\in\mathcal{L}_1$ of type $t(\Lambda)$.

Now assume $t(\Lambda)=t_0$. Then $\Lambda^{(1)}$ is a vertex of type $t_0$, i.e.\ of maximal type, in $\mathcal{L}_1$. Using Lemma \ref{einslambda} on special homomorphisms of valuation 1, we have for $m_1<i\leq m_2$:
\[ (\Lambda_\bullet^{(1)})\cap(W_\mathbb{Q}x_i\oplus W_\mathbb{Q}F^{-1}x_i)=p^{-1}Wx_i\oplus WF^{-1}x_i=L_i, \]
which is again self-dual w.r.t.\ $\langle.,.\rangle$. Again, we use Remark \ref{thirtysix} to get the orthogonal decomposition
\[ \Lambda_\bullet=\left(\bigoplus_{i=1}^{m_1}L_i\right)\oplus\Lambda_\bullet^{(1)}=\left(\bigoplus_{i=1}^{m_2}L_i\right)\oplus\Lambda_\bullet^{(2)} \]
for some $\Lambda^{(2)}\in\mathcal{L}_2$ of type $t_0-m_2+m_1$. But the right hand side is precisely $\left(\Phi_1(\Lambda^{(2)})\right)_\bullet$, and $\Lambda^{(2)}$ is in $T_1'$ because $\Lambda$ is in $T_0$. Thus by assumption, $\Lambda\in\Phi_1(V_1')$, which is a subset of $V_0$ by definition. This proves the lemma.
\end{proof}

\begin{lem} \label{knuddel}
Let $0\leq s\leq\lfloor r_m/2\rfloor$ be arbitrary. Assume that $S_{s+1}'=V_{s+1}'$. Then $S_s=V_s$.
\end{lem}

\begin{proof}
We observe that $S_{s+1}'=\{\Lambda\in\mathcal{L}_{2s+2}\ |\ p^{-s}x_i\in\Lambda^\sharp\ \forall\,i>m_{2s+2}\}$ is the simplicial complex $S_1'(\underline{z})$ obtained by applying Def.\ \ref{allesdoof} to
\[ \underline{z}:=(p^{-\lfloor\frac{r_1}{2}\rfloor}x_1,\dots,p^{-\lfloor\frac{r_{m_{2s+2}}}{2}\rfloor}x_{m_{2s+2}},p^{-s}x_{m_{2s+2}+1},\dots,p^{-s}x_m). \]
The morphism $\Phi_0\circ\dots\circ\Phi_{s-1}:\mathcal{L}_{2s}\rightarrow\mathcal{L}$ of simplicial complexes is injective and, by definition of $\underline{z}$, maps $S_s$ into $\mathcal{S}(\underline{z})$, sending vertices of type $t$ in $\mathcal{L}_{2s}$ to vertices of type $t+\sum_{i=1}^{2s}(-1)^im_i$ in $\mathcal{L}$.

On the other hand, by pure-dimensionality (Prop.\ \ref{dimension}), there is an odd integer $t_0$ such that any vertex in $\mathcal{S}(\underline{z})$ is contained in some vertex in $\mathcal{S}(\underline{z})^{(t_0)}$. By the proof of Lemma \ref{fluff}, we have
\[ \mathcal{S}(\underline{z})^{(t_0)}=(\Phi_0\circ\dots\circ\Phi_s)\left(S_{s+1}'\cap\mathcal{L}_{2s+2}^{(t_0-\sum_{i=1}^{2s+2}(-1)^im_i)}\right). \]
Since $\Phi_0\circ\dots\circ\Phi_s:S_{s+1}'\rightarrow\mathcal{S}(\underline{z})$ factors over $S_s$ and $\Phi_0\circ\dots\circ\Phi_{s-1}:\mathcal{L}_{2s}\rightarrow\mathcal{L}$ is injective, it follows that we have the expected ``pure-dimensionality statement'' for $S_s$, namely that every vertex in $S_s$ is contained in a vertex in $S_s\cap\mathcal{L}_{2s}^{(t_0-\sum_{i=1}^{2s}(-1)^im_i)}$, and that
\[ S_s\cap\mathcal{L}_{2s}^{(t_0-\sum_{i=1}^{2s}(-1)^im_i)}=\Phi_s\left(\mathcal{S}_{s+1}'\cap\mathcal{L}_{2s+2}^{(m_{2s-2}-m_{2s-1})}\right). \]
Therefore, $S_{s+1}'=V_{s+1}'$ implies $S_s=V_s$.
\end{proof}

\begin{lem} \label{schmieg}
Let $1\leq s\leq\lfloor r_m/2\rfloor$. Assume that $S_s=V_s$. Then $V_s'\subseteq S_s'$.
\end{lem}

\begin{proof}
First, note that by definition of $V_s'$, any vertex in $V_s'$ is contained in some vertex $\Lambda$ of maximal type in $\mathcal{L}_{2s}$ for which $d_{2s}(\Lambda,V_s)\leq 1$. Therefore it suffices to show that any $\Lambda\in\mathcal{L}_{2s}^{\max}$ for which $d_{2s}(\Lambda,V_s)\leq 1$ is in $S_s'$. Let $\Lambda$ be such a vertex. By assumption $V_s=S_s$, thus $\Lambda$ satisfies $d_{2s}(\Lambda,S_s)\leq 1$, which means that there is some vertex $\tilde\Lambda\in S_s$ contained in $\Lambda$. Thus
\[ p^{-s+1}x_i\in p\tilde\Lambda^\sharp\subseteq p\tilde\Lambda\subseteq p\Lambda\subseteq\Lambda^\sharp\quad\forall\ i>m_{2s}. \]
\end{proof}

\begin{lem} \label{kuschel}
Let $1\leq s\leq\lfloor r_m/2\rfloor$. Assume that $S_s=V_s$. Then $S_s'$ is ``pure of codimension 0'' in $\mathcal{L}_{2s}$, that is, every vertex in $S_s'$ is contained in some vertex in $S_s'\cap\mathcal{L}_{2s}^{\max}$.
\end{lem}

\begin{proof}
The rough idea is, as in the proof of Lemma \ref{knuddel}, to use Lemma \ref{fluff} on some variation of $\underline{x}$ to reduce to the known case of $S_0$. In our situation, we have $S_s'=S_1'(\underline{z})$, where
\[ \underline{z}=(p^{-\lfloor\frac{r_1}{2}\rfloor}x_1,\dots,p^{-\lfloor\frac{r_{m_{2s}}}{2}\rfloor}x_{m_{2s}},p^{-s+1}x_{m_{2s}+1},\dots,p^{-s+1}x_m). \]
By Lemma \ref{dimension}, every vertex in $\mathcal{S}(\underline{z})$ is contained in some vertex of type $t_0$, where $t_0-\sum_{i=1}^{2s}(-1)^im_i$ is the maximal type of vertices in $\mathcal{L}_{2s}$. But we can describe $S_s'$ in terms of $S_0(\underline{z})=\mathcal{S}(\underline{z})$ in the same way as we did in the proof of Lemma \ref{knuddel}. The result is an injection of simplicial complexes
\[ \Phi_0\circ\dots\circ\Phi_{s-1}:S_s'\rightarrow S_0(\overline{z}) \]
sending vertices in $\mathcal{L}_{2s}^{(t)}$ to vertices in $\mathcal{L}^{(t+\sum_{i=1}^{2s}(-1)^im_i)}$, and surjective in maximal type. Therefore, the ``pure-dimensionality'' statement Prop.\ \ref{dimension} for $\mathcal{S}(\underline{z})$ implies the claimed ``pure-dimensionality'' statement for $S_s'$.
\end{proof}

\begin{lem} \label{schnuff}
Let $1\leq s\leq\lfloor r_m/2\rfloor$. Assume that $S_s=V_s$. Let $\Lambda\in S_s'$ be of maximal type in $\mathcal{L}_{2s}$. Then $\Lambda\in V_s'$; that is, there is some vertex $\tilde\Lambda\in V_s=S_s$ contained in $\Lambda$.
\end{lem}

\begin{remark}
This lemma is a generalization of Lemma \ref{upperbound} to the case of $m>1$ and arbitrary order of determinant of $h$. However, unlike in the $m=1$ case, we do not have any lower bound on the maximal possible type of $\tilde\Lambda$. Indeed, $S_s$ can be arbitrarily small, even consist of a single vertex of type 0.
\end{remark}

\begin{proof}
Let $\Lambda\in S_s'$ be of maximal type in $\mathcal{L}_{2s}$. The type of $\Lambda$ is either $\dim(C_{2s})$ or $\dim(C_{2s})-1$, depending on the parity of $\dim(C_{2s})$ and on the parity of the order of determinant of $(C_{2s},h)$. 

Let $m_{2s}<i\leq m$. By assumption, we have $p^{-s+1}x_i\in\Lambda^\sharp$. This implies (using that $r(x_i)>2s$) that $p^{-s}x_i\in\Lambda$. Indeed, Lemma \ref{tomate} generalizes to our situation, since the proof does not depend on the order of determinant of $h$ being odd (except for that, if $\textnormal{ord}\,\textnormal{det}(C,h)$ was even, the case of $n$ even would be the easy one to prove and the case of $n$ odd the hard one).

We now proceed as in the proof of Lemma \ref{upperbound}. We consider the hermitian $\mathbb{F}_{p^2}$-vector space $(V,\overline{ph})$, where $V=\Lambda/\Lambda^\sharp$ and $\overline{ph}$ is induced by $ph$. By \ref{muff}, vertices $\tilde\Lambda$ of type $t$ contained in $\Lambda$ correspond to $((t(\Lambda)-t)/2)$-dimensional isotropic subspaces $U=\tilde\Lambda^\sharp/\Lambda^\sharp$ of $V$. Furthermore, $p^{-s}x_i\in\tilde\Lambda^\sharp$ is equivalent to $\overline{p^{-s}x_i}\in U$, where $\overline{p^{-s}x_i}$ is the image of $p^{-s}x_i$ in $V$.

Now the $\overline{p^{-s}x_i}$ are isotropic (since $v(ph(p^{-s}x_i,p^{-s}x_i))=2r_i-2s+1>0$) and perpendicular to each other (since the $x_i$ were already assumed to be perpendicular to each other w.r.t.\ $h$). Therefore, they span an isotropic subspace $U$ of $V$, whose preimage $\tilde\Lambda^\sharp$ is the dual of some vertex of type $t(\Lambda)-2\dim U$ in $\mathcal{L}_{2s}$ belonging to $S_s$ and contained in $\Lambda$.
\end{proof}
This concludes the proof of Thm.\ \ref{vielex} and thus proves Thm.\ \ref{vielexprime}.

\begin{fliess}
We now prove our main theorem \ref{holygrail} in full generality.

First we observe that by Bruhat-Tits theory, the simplicial complex $\mathcal{L}_{\lfloor r_m/2\rfloor+1}\cong\mathfrak{B}(SU(C_{\lfloor r_m/2\rfloor+1},ph),\mathbb{Q}_p)$ is connected. Thm.\ \ref{vielex} provides us with an algorithm for computing $\mathcal{S}(\underline{x})=S_0$ from $\mathcal{L}_{\lfloor r_m/2\rfloor+1}$. To show connectedness of $\mathcal{S}(\underline{x})$, it is enough to prove that each step of this algorithm preserves connectedness.

Let $0\leq s\leq\lfloor r_m/2\rfloor$. We want to show that $S_s$ is connected under the assumption that $S_{s+1}'$ is connected. By Thm.\ \ref{vielex}, we have
\[ S_s:=\{\Lambda\in\mathcal{L}_{2s}\ |\ \exists\ \tilde\Lambda\in \Phi_s(S_{s+1}'):\ \Lambda\subseteq\tilde\Lambda\}. \]
Now $\Phi_s$ is a morphism of simplicial complexes, thus preserves connectedness, therefore $\Phi_s(S_{s+1}')$ is connected. But any vertex in $S_s$ is contained in (i.e.\ a neighbour of) some vertex in $\Phi_s(S_{s+1}')$. Thus passing from $\Phi_s(S_{s+1}')$ to $S_s$ does not create new connected components.

Now let $1\leq s\leq\lfloor r_m/2\rfloor$. We assume that $S_s$ is connected. Recall that by Thm.\ \ref{vielex}:
\[ S_s':=\{\Lambda\in\mathcal{L}_{2s}\ |\ \exists\ \tilde\Lambda\in\mathcal{L}_{2s}^{\max}:\ \Lambda\subseteq\tilde\Lambda,\,d_{2s}(\tilde\Lambda,S_s)\leq 1\}. \]
Connectedness of $S_s$ now implies connectedness of $S_s'$. Indeed, we saw in Section 8 how connectedness of $\mathcal{S}(p^{-1}x)$ implies connectedness of $\mathcal{S}(x)$ for a single special homomorphism $x$. The argument for this implication generalizes straightforwardly to our situation. This concludes the proof of Thm.\ \ref{holygrail}
\end{fliess}
As a further application of Thm.\ \ref{vielex}, we will now give a new proof of the criterion of Kudla and Rapoport for irreducibility in the case $m=n$, which will highlight the role of Bruhat-Tits theory.

\begin{prop}[Kudla-Rapoport, \cite{kr}, Thm.\ 1.1(iii)]
Let $m=n$. Let $\underline{x}\in\mathbb{V}^n$ be according to the assumptions of \ref{fleshwound} and the notational conventions of this section. Let
\begin{align*}
n^+_\textnormal{even} & =\textnormal{Card}(\{i\ |\ r_i\geq 2,\,r_i\ \textnormal{even}\}),\\
n^+_\textnormal{odd} & =\textnormal{Card}(\{i\ |\ r_i\geq 3,\,r_i\ \textnormal{odd}\}).
\end{align*}
Then $\mathcal{Z}(\underline{x})_\textnormal{red}$ is irreducible if and only if
\[ \max(n^+_\textnormal{even},n^+_\textnormal{odd})\leq 1. \]
\end{prop}

\begin{proof}
We first observe that $\mathcal{Z}(\underline{x})$ is irreducible if and only if there is $\Lambda\in\mathcal{L}$ for which $\mathcal{Z}(\underline{x})_\textnormal{red}=\mathcal{N}_\Lambda$. The latter statement means that $\Lambda\in\mathcal{S}(\underline{x})$ and that every $\tilde\Lambda\in\mathcal{S}(\underline{x})$ is contained in $\Lambda$. In other words, $\mathcal{Z}(\underline{x})_\textnormal{red}$ is irreducible if and only if $\mathcal{S}(\underline{x})$ has a unique element maximal for inclusion.

Now take a look at the behaviour of ``irreducible components'' under the operations of the algorithm of Thm.\ \ref{vielex}. First, the simplicial complex $\mathcal{L}_{r_n+1}$ consists of a single point, because we assumed $m=n$ and thus $C_{r_n+1}$ is the zero space. Therefore, $S_{\lfloor r_n/2\rfloor+1}'$ consists of a single point, thus is trivially ``irreducible''.

Then, for any $s$, passing from $S_{s+1}'$ to $S_s$ does not create or destroy maximal elements. Indeed, $\Phi_s$ is an injection of simplicial complexes which preserves inclusions, i.e.\ $\tilde\Lambda\subseteq\Lambda\Rightarrow\Phi_s(\tilde\Lambda)\subseteq\Phi_s(\Lambda)$. Therefore, maximal elements $\Lambda$ in $S_{s+1}'$ give rise to maximal elements $\Phi_s(\Lambda)$ in $\Phi_s(S_{s+1}')$, and no new maximal elements are generated. Also, any element in $S_s$ is contained in some element in $\Phi_s(S_{s+1}')$, hence the maximal elements in $S_s$ are the same as those in $\Phi_s(S_{s+1}')$.

On the other hand, passing from $S_s$ to $S_s'$ creates new inclusion maximal elements whenever $\mathcal{L}_{2s}$ is not a single point. Indeed, if $\mathcal{L}_{2s}$ is not a single point, then different types of vertices occur (e.g.\ any two vertices which neighbour each other are of different type). Let $\Lambda\in S_s$ be arbitrary, not of maximal type in $\mathcal{L}_{2s}$. Then there is more than one vertex of maximal type containing $\Lambda$. Indeed, $\Lambda$ can be written as the intersection of the vertices of maximal type containing it (see the proof of Prop.\ \ref{beliebigekodim}). By Thm.\ \ref{vielex}, all such vertices are in $S_s'$, and maximal because they are maximal in $\mathcal{L}_{2s}$.

Therefore, $S_0=\mathcal{S}(\underline{x})$ has a unique maximal element if and only if the $\mathcal{L}_{2s}$ for $s\geq 1$ are single points. The latter is equivalent to saying that $\mathcal{L}_2$ consists of a single point, since $\Phi_{s-1}\circ\dots\circ\Phi_1:\mathcal{L}_{2s}\rightarrow\mathcal{L}_2$ is injective. But $\mathcal{L}_2\cong\mathfrak{B}(SU(C_2,ph),\mathbb{Q}_p)$ consists of a single point if and only if either $\dim C_2\leq 1$, or $\dim C_2=2$ and the order of determinant of $h$ on $C_2$ is odd.

Since $C_2$ is by definition spanned by the $x_i$ with $r_i>1$, $\dim C_2=0$ means that $n^+_\textnormal{odd}=n^+_\textnormal{even}=0$, while $\dim C_2=1$ means that one of the sums $n^+_\textnormal{odd}$, $n^+_\textnormal{even}$ equals 1, while the other one vanishes, and finally $\dim C_2=2$ together with the assumption that $\textnormal{ord}\,\textnormal{det}(C_2,h)$ is odd means that $n^+_\textnormal{odd}=n^+_\textnormal{even}=1$.

\end{proof}


\begin{thebibliography}{ABC}
\bibitem[Ja]{jack} Jacobowitz, R.: Hermitian forms over local fields. Amer.\ J.\ Math.\ \bfseries84\normalfont, 441-465 (1962)
\bibitem[KR]{kr} Kudla, S., Rapoport, M.: Special cycles on unitary Shimura varieties. I: Unramified local theory. Invent.\ Math.\ \bfseries184\normalfont, 629-682 (2011)
\bibitem[KR2]{kr3} Kudla, S., Rapoport, M.: Special cycles on unitary Shimura varieties. II: Global theory. arXiv: 0912.3758v1
\bibitem[KR3]{kr2} Kudla, S., Rapoport, M.: The alternative moduli problem (unpublished notes), 2010
\bibitem[Te]{ter} Terstiege, U.: Intersections of special cycles on the Shimura variety for $GU(1,2)$. arXiv: 1006.2106
\bibitem[Ti]{ti} Tits, J.: Reductive groups over local fields. In: Automorphic forms, representations and $L$-functions. Proc.\ Sympos.\ Pure Math.\ \bfseries33\normalfont, 29-69 (1979)
\bibitem[Vo]{voll} Vollaard, I.: The supersingular locus of the Shimura variety for $GU(1,s)$. Can.\ J.\ Math.\ \bfseries62\normalfont, 668-720 (2010)
\bibitem[VW]{vw} Vollaard, I., Wedhorn, T.: The supersingular locus of the Shimura variety for $GU(1,n-1)$ II. Invent.\ Math.\ \bfseries184\normalfont, 591-627 (2010).
\end{thebibliography}
\end{document}